\documentclass[11pt,reqno]{amsart}
\usepackage{amsthm}
\usepackage{amssymb}
\usepackage{latexsym}
\usepackage{multicol}
\usepackage{verbatim,enumerate}
\usepackage{amscd}

\usepackage[usenames]{color}
\usepackage{hyperref}
\usepackage{amsmath, amscd}

\advance\textwidth by 1.2in \advance\oddsidemargin by -.6in \advance\evensidemargin by -.6in
\newcommand{\dom}[1]{\text{\rm $#1$-dom}}
\newcommand{\ud}[1]{\underline{#1}}
\newcommand{\BB}{\mathbb{B}}
\newcommand{\BT}{\mathbb{T}}
\newtheorem*{cor}{Corollary}%[section]
\newtheorem*{lem}{Lemma}
\newtheorem*{prop}{Proposition}

\theoremstyle{definition}
\newtheorem*{defn}{Definition}
\theoremstyle{definition}
\newtheorem{thm}{Theorem}
\newtheorem*{thm*}{Theorem}
\newtheorem*{conj}{Conjecture}

\newenvironment{pf}{\proof}{\endproof}
\newcounter{cnt}
\newenvironment{enumerit}{\begin{list}{{\hfill\rm(\roman{cnt})\hfill}}{%
\settowidth{\labelwidth}{{\rm(iv)}}\leftmargin=\labelwidth%
\advance\leftmargin by \labelsep\rightmargin=0pt\usecounter{cnt}}}{\end{list}} \makeatletter
\def\mydggeometry{\makeatletter\dg@YGRID=1\dg@XGRID=20\unitlength=0.003pt\makeatother}
\makeatother \theoremstyle{remark}

\numberwithin{equation}{section}

\let\bwdg\bigwedge
\def\bigwedge{{\textstyle\bwdg}}

\newcommand{\ti}[1]{\widetilde{#1}}

\begin{document}

\newcommand{\thmref}[1]{Theorem~\ref{#1}}
\newcommand{\secref}[1]{Section~\ref{#1}}
\newcommand{\lemref}[1]{Lemma~\ref{#1}}
\newcommand{\propref}[1]{Proposition~\ref{#1}}
\newcommand{\corref}[1]{Corollary~\ref{#1}}
\newcommand{\remref}[1]{Remark~\ref{#1}}
\newcommand{\defref}[1]{Definition~\ref{#1}}
\newcommand{\er}[1]{(\ref{#1})}
\newcommand{\id}{\operatorname{id}}
\newcommand{\ord}{\operatorname{\emph{ord}}}
\newcommand{\sgn}{\operatorname{sgn}}
\newcommand{\wt}{\operatorname{wt}}
\newcommand{\tensor}{\otimes}
\newcommand{\from}{\leftarrow}
\newcommand{\nc}{\newcommand}
\newcommand{\rnc}{\renewcommand}
\newcommand{\dist}{\operatorname{dist}}
\newcommand{\qbinom}[2]{\genfrac[]{0pt}0{#1}{#2}}
\nc{\cal}{\mathcal} \nc{\goth}{\mathfrak} \rnc{\bold}{\mathbf}
\renewcommand{\frak}{\mathfrak}
\newcommand{\supp}{\operatorname{supp}}
\newcommand{\Irr}{\operatorname{Irr}}
\newcommand{\psym}{\mathcal{P}^+_{K,n}}
\newcommand{\psyml}{\mathcal{P}^+_{K,\lambda}}
\newcommand{\psymt}{\mathcal{P}^+_{2,\lambda}}
\renewcommand{\Bbb}{\mathbb}
\nc\bomega{{\mbox{\boldmath $\omega$}}} \nc\bpsi{{\mbox{\boldmath $\Psi$}}}
 \nc\balpha{{\mbox{\boldmath $\alpha$}}}
 \nc\bpi{{\mbox{\boldmath $\pi$}}}
\nc\bmu{{\mbox{\boldmath $\mu$}}} \nc\bcN{{\mbox{\boldmath $\cal{N}$}}} \nc\bcm{{\mbox{\boldmath $\cal{M}$}}} \nc\blambda{{\mbox{\boldmath
$\lambda$}}}\nc\bnu{{\mbox{\boldmath $\nu$}}}

\newcommand{\Tmn}{\bold{T}_{\lambda^1, \lambda^2}^{\nu}}

\newcommand{\lie}[1]{\mathfrak{#1}}
\newcommand{\ol}[1]{\overline{#1}}
\makeatletter
\def\section{\def\@secnumfont{\mdseries}\@startsection{section}{1}%
  \z@{.7\linespacing\@plus\linespacing}{.5\linespacing}%
  {\normalfont\scshape\centering}}
\def\subsection{\def\@secnumfont{\bfseries}\@startsection{subsection}{2}%
  {\parindent}{.5\linespacing\@plus.7\linespacing}{-.5em}%
  {\normalfont\bfseries}}
\makeatother
\def\subl#1{\subsection{}\label{#1}}
 \nc{\Hom}{\operatorname{Hom}}
  \nc{\mode}{\operatorname{mod}}
\nc{\End}{\operatorname{End}} \nc{\wh}[1]{\widehat{#1}} \nc{\Ext}{\operatorname{Ext}} \nc{\ch}{\text{ch}} \nc{\ev}{\operatorname{ev}}
\nc{\Ob}{\operatorname{Ob}} \nc{\soc}{\operatorname{soc}} \nc{\rad}{\operatorname{rad}} \nc{\head}{\operatorname{head}}
\def\Im{\operatorname{Im}}
\def\gr{\operatorname{gr}}
\def\mult{\operatorname{mult}}
\def\Max{\operatorname{Max}}
\def\ann{\operatorname{Ann}}
\def\sym{\operatorname{sym}}
\def\Res{\operatorname{\br^\lambda_A}}
\def\und{\underline}
\def\Lietg{$A_k(\lie{g})(\bsigma,r)$}

 \nc{\Cal}{\cal} \nc{\Xp}[1]{X^+(#1)} \nc{\Xm}[1]{X^-(#1)}
\nc{\on}{\operatorname} \nc{\Z}{{\bold Z}} \nc{\J}{{\cal J}} \nc{\C}{{\bold C}} \nc{\Q}{{\bold Q}}
\renewcommand{\P}{{\cal P}}
\nc{\N}{{\Bbb N}} \nc\boa{\bold a} \nc\bob{\bold b} \nc\boc{\bold c} \nc\bod{\bold d} \nc\boe{\bold e} \nc\bof{\bold f} \nc\bog{\bold g}
\nc\boh{\bold h} \nc\boi{\bold i} \nc\boj{\bold j} \nc\bok{\bold k} \nc\bol{\bold l} \nc\bom{\bold m} \nc\bon{\bold n} \nc\boo{\bold o}
\nc\bop{\bold p} \nc\boq{\bold q} \nc\bor{\bold r} \nc\bos{\bold s} \nc\boT{\bold t} \nc\boF{\bold F} \nc\bou{\bold u} \nc\bov{\bold v}
\nc\bow{\bold w} \nc\boz{\bold z} \nc\boy{\bold y} \nc\ba{\bold A} \nc\bb{\bold B} \nc\bc{\bold C} \nc\bd{\bold D} \nc\be{\bold E} \nc\bg{\bold
G} \nc\bh{\bold H} \nc\bi{\bold I} \nc\bj{\bold J} \nc\bk{\bold K} \nc\bl{\bold L} \nc\bm{\bold M} \nc\bn{\bold N} \nc\bo{\bold O} \nc\bp{\bold
P} \nc\bq{\bold Q} \nc\br{\bold R} \nc\bs{\bold S} \nc\bt{\bold T} \nc\bu{\bold U} \nc\bv{\bold V} \nc\bw{\bold W} \nc\bz{\bold Z} \nc\bx{\bold
x} \nc\KR{\bold{KR}} \nc\rk{\bold{rk}} \nc\het{\text{ht }}

\nc\toa{\tilde a} \nc\tob{\tilde b} \nc\toc{\tilde c} \nc\tod{\tilde d} \nc\toe{\tilde e} \nc\tof{\tilde f} \nc\tog{\tilde g} \nc\toh{\tilde h}
\nc\toi{\tilde i} \nc\toj{\tilde j} \nc\tok{\tilde k} \nc\tol{\tilde l} \nc\tom{\tilde m} \nc\ton{\tilde n} \nc\too{\tilde o} \nc\toq{\tilde q}
\nc\tor{\tilde r} \nc\tos{\tilde s} \nc\toT{\tilde t} \nc\tou{\tilde u} \nc\tov{\tilde v} \nc\tow{\tilde w} \nc\toz{\tilde z} \nc\woi{w_{\omega_i}}
\nc\chara{\operatorname{Char}}
\author[Chari, Fourier and Sagaki]{Vyjayanthi Chari, Ghislain Fourier and Daisuke Sagaki}
\address{ Department of Mathematics, University of California, Riverside, CA 92521}
\email{vyjayanthi.chari@ucr.edu}
\address{ Mathematisches Institut, Universit\"at zu K\"oln, Germany}
\email{gfourier@math.uni-koeln.de}
\address{ Department of Mathematics, University of Tsukuba, Japan}
\email{sagaki@math.tsukuba.ac.jp}

\thanks{\noindent V.C. was partially supported by DMS-0901253.\\
\noindent  G.F was partially supported by the DFG priority program 1388 ''Representation Theory'' \\
\noindent D.S. was partially supported by the Grant-in-Aid for Young Scientists (B) No.23740003
}

%\subjclass[2000]{Primary: 14C05,17B69}
\date{\today}

\title{Posets, Tensor Products and Schur positivity }
\begin{abstract}

Let $\lie g$ be  a complex finite-dimensional simple Lie algebra.
Given a positive integer $k$ and a dominant weight $\lambda$,
we define a preorder  $\preceq$ on the set $P^{+}(\lambda,\,k)$ of
$k$-tuples of dominant weights which add up to $\lambda$.
Let $\sim$ be the  equivalence relation defined by the preorder and
 $P^{+}(\lambda,\,k)/\!\sim$ be the corresponding poset of equivalence classes.
We show that if $\lambda$ is a multiple of
a fundamental weight (and $k$ is general) or
if $k=2$ (and $\lambda$ is general), then
$P^{+}(\lambda,\,k)/\!\sim$ coincides with the set of
$S_{k}$-orbits in $P^{+}(\lambda,\,k)$,
where $S_{k}$ acts on $P^{+}(\lambda,\,k)$ as the permutations of components.
If  $\lie g$ is of type $A_n$ and $k=2$,
we show that the $S_{2}$-orbit of
the row shuffle defined by Fomin et al in \cite{FFLP2005} is the unique maximal element in the poset.

Given an element of $P^{+}(\lambda,\,k)$,
consider the tensor product of the corresponding
simple finite-dimensional $\lie g$-modules.
We show that (for general $\lie g$, $\lambda$, and $k$)
the dimension of this tensor product increases along $\preceq$.
We also show that in the case when
$\lambda$ is a multiple of a fundamental minuscule weight
($\lie g$ and $k$ are general) or
if $\lie g$ is of type $A_{2}$ and $k=2$ ($\lambda$ is general),
there exists an inclusion of tensor products along with the partial
order $\preceq$ on $P^{+}(\lambda,\,k)/\!\sim$. In particular,  if $\lie g$
is of type $A_n$,  this means that the difference of
the characters is Schur positive.
\end{abstract}

\maketitle

%%%%%%%%%%%%%%%%%%%%%%%%%%%%%%%%%%%%%%%%%%%%%%%%%%%%%%%%%%%%%%%%%%
%%%%%%%%%%%%%%%%%%%%%%%%%%%Introduction%%%%%%%%%%%%%%%%%%%%%%%%%%%

\section*{Introduction}
This paper is partially motivated by the following simple observation. The isomorphism classes of simple finite-dimensional representations of the complex simple Lie algebra $\lie{sl}_2$ are indexed by $\bz_+$, the set non-negative integers.  Given $r\in\bz_+$, let $V(r)$ be a representative of the corresponding isomorphism class.  The well-known Clebsch-Gordan formula gives us the following decomposition: for $r,s\in\bz_+$,
$$V(r) \otimes V(s) \cong V(r+s) \oplus V(r+s-2) \oplus \cdots \oplus V(|r-s|).$$
If $r_1,s_1\in\bz_+$ are such that $r_1+s_1=r+s$, it is then immediate that
one has an injection of $\lie{sl}_2$-modules,
\begin{equation} \label{surjection}
V(r) \otimes V(s) \hookrightarrow V(r_1) \otimes V(s_1) \iff |r-s| \geq |r_1 - s_1|.
\end{equation}
In particular, the dimension increases if $\min\{r, s\} \leq \min\{r_1,s_1\}$. Moreover the pairs $(r,r)$ and $(r,r+1)$ are maximal in the sense that: the corresponding tensor products maps onto to any tensor product corresponding to $(r_1,s_1)$ if $r_1+s_1=2r$ (resp. $r_1+s_1=2r+1$).  Notice that these maximal pairs are actually the simplest examples of the row shuffle of partitions given in \cite{FFLP2005}. Part of our interest in this problem comes from the fact that the tensor product $V(r)\otimes V(s)$ admits the structure of an indecomposable  module for the infinite-dimensional Lie algebra $\lie{sl}_2\otimes \bc[t]$ where $\bc[t]$ is the polynomial ring in the indeterminate $t$. In this case, the map in \eqref{surjection} is actually a map of $\lie{sl}_2\otimes \bc[t]$-modules and the module corresponding to the maximal pair is a truncation of a local Weyl module. We shall return to these ideas elsewhere. We mention this  only to indicate our original motivation,   the  results of the current paper are entirely about simple Lie algebras.

Suppose now that $\lie g$ is a finite-dimensional simple Lie algebra, $P^+$ the set of dominant integral weights. We let $P^+(\lambda,2)$ be the set of ``compositions'' of $\lambda$ with at most two parts: i.e. pairs of dominant integral weights which add up to $\lambda$. If $(\lambda_1,\lambda_2)$ and $(\mu_1,\mu_2)$ are two such compositions, we define a partial order $\preceq$ by requiring that $\min\{\lambda_1(h_\alpha),\lambda_2(h_\alpha)\}\le \min\{\mu_1(h_\alpha),\mu_2(h_\alpha)\}$ holds for all positive roots $\alpha$. This order extends in a natural way to $P^+(\lambda, k)$: the set compositions of $\lambda$ with $k$ parts for all $k\ge 1$, one just requires the inequality to hold for all partial sums.

For $\mu\in P^+$, let $V(\mu)$ be the corresponding finite-dimensional $\lie g$-module. If $\blambda$ and $\bmu$ are two compositions of $\lambda$ with $k$ parts and $V(\blambda)$, $V(\bmu)$ are the corresponding tensor products of $\lie g$-modules, we prove that
$$\blambda\preceq\bmu\implies \dim V(\blambda)\le\dim V(\bmu).$$
If $\lambda$ is a multiple of a minuscule weight, then we prove that there exists a (non-canonical) inclusion $V(\blambda)\hookrightarrow V(\bmu)$. In the case when $\lie g$ is of type $\lie{sl}_3$ and $k=2$ we prove that the inclusion holds for all $\lambda\in P^+$. We conjecture that this latter result holds for all simple Lie algebras.
Our  conjecture may be viewed as a generalization of the row-shuffle conjecture which was  made in  \cite{FFLP2005} for representations of $\lie{sl}_{n+1}$. The row shuffle conjecture  and proved in  \cite{LPP2007} and shows that our conjecture is true in the case of $\lie{sl}_{n+1}$ for the pair $\blambda,\bmu$ where $\bmu$ is the maximal element in the poset.
A completely different approach  taken in \cite{DP2007} also gives some evidence for our conjecture in the case of $\lie{sl}_{n+1}$.

 \vskip 12pt
 
 \noindent The following is  an immediate combinatorial consequence of Theorem~\ref{mainthm} (i) and is perhaps of independent interest.
Given two partitions
$(0 \leq a_1 \leq a_2 \leq \ldots \leq a_k)$ and $ (0 \leq b_1 \leq b_2 \leq \ldots \leq b_k)$  of an integer $n$  satisfying:
 $$\sum_{j = 1}^i a_j \leq \sum_{j = 1}^i b_j \text{ for all } 1 \leq i \leq k .$$ Then $$ a_1\cdots a_k \le  b_1\cdots b_k$$ with equality
iff
$a_j= b_j \text{ for all  } 1 \leq j \leq k $.

\vskip 12pt

The article is organized as follows. Section~\ref{section1} has the basic definitions and notation.
In Section~\ref{section2} we introduce the partial order and state the main theorem.
In Section~\ref{section3} we prove that   the dimension of the tensor product increases along the partial order.
The critical idea in this proof is to show that in the case of $\lie{sl}_2$,
the partial order for $k>2$ is determined by the order at $k=2$.
Once this is done, the proof is a simple application of the Weyl dimension formula. 
In Section~\ref{section4} we use the results of Section~\ref{section2} and
the Littelmann path model to prove that for general $k$ and $\lambda$ a multiple of
a minuscule weight, we have an inclusion of tensor products along the partial order.
In Section~\ref{section5} we  study the partial order $P^+(\lambda,2)$ in detail.
We identify maximal elements of the poset and prove that the row shuffle is
the unique maximal element when $\lie g$ is of type $\lie{sl}_{n+1}$.
%
% We also identify a set of all possible covers in the partial order.
%
Finally, in Section~\ref{section6}, we use results of Kashiwara and
Nakashima on semistandard Young tableaux and crystal bases to prove that
in the case that $\lie{g}=\lie{sl}_3$ and $k=2$,
the Schur positivity holds along our order for all $\lambda\in P^{+}$.

\textbf{Acknowledgements:} The authors thank the Hausdorff  Institute for Mathematics for the invitation to participate in the trimester program ''On the Interaction of Representation Theory with Geometry and Combinatorics'' where this collaboration was initiated. The authors would also thank Anne Schilling for helpful discussions.

%%%%%%%%%%%%%%%%%%%%%%%%%%%%%%%%%%%%%%%%%%%%%%%%%%%%%%%%%%%%%%%%%%%%%%%%%%%%%%%%%%%%%%%%%%%%%%%%%%%%
%%%%%%%%%%%%%%%%%%%%%%% Notations and preliminaries

\section{Preliminaries}\label{section1}

 Throughout this paper we denote by   $\bc$  the field of complex numbers and $\bz$ (resp. $\bz_+$) the set of integers (resp.  nonnegative integers).

\subsection{}
Let  $\lie g$  denote a finite-dimensional complex simple Lie algebra of rank $n$ and $\lie h$   a fixed Cartan subalgebra of $\lie g$. Let   $I=\{1,\cdots ,n\}$   and fix a set  $\{\alpha_i: i\in I\}$ of simple roots of $\lie g$ with respect to $\lie h$ and a set $\{\omega_i: i\in I\}$   of fundamental weights.  Let $P$ (resp. $P^+$)  be the $\bz$ (resp. $\bz_+$) span of  $\{\omega_i: i\in I\}$.  Let $R$ and $R^+ $ be the set of roots and  positive  roots  respectively. For $\alpha\in R^+$, let $\lie{sl}_2(\alpha) = \langle x_{\alpha}^{\pm}, h_{\alpha} \rangle$ be the corresponding subalgebra of $\lie g$ isomorphic to $\lie{sl}_2$ and set $h_{-\alpha}=-h_\alpha$ for $\alpha\in R^+$, $h_i = h_{\alpha_i}$.  Let $W$ be the Weyl group of $\lie g$ and recall that  $W$ acts on $\lie h$ and $\lie h^*$ and that for all $\bow\in W$, $\lambda\in\lie h^*$ and $\alpha\in R^+$, we have  $$(\bow\lambda)(\bow h_\alpha)=\lambda(h_\alpha).$$  For $\alpha\in R^+$, let $\bos_\alpha\in W$ be the corresponding reflection and we set $\bos_i=\bos_{\alpha_i}$.  Let $\bow_0\in W$ be the longest element.

We say that $\lambda\in P^+$ is minuscule if $\lambda(h_\alpha) \in \{ 0,1 \}$
for all $\alpha \in R^+$. It can be easily seen that if $\lambda \in P^{+}$
is minuscule, then $\lambda$ is a fundamental weight.
The following is the list of minuscule weights; here, we follow the numbering
of vertices of the Dynkin diagram for $\lie{g}$ in \cite{Bour}.
\begin{equation*}
\begin{array}{ll}
A_{n} & \omega_{i},\, 1 \le i \le n \\[1mm]
B_{n} & \omega_{n} \\[1mm]
C_{n} & \omega_{1} \\[1mm]
D_{n} & \omega_{1},\,\omega_{n-1},\,\omega_{n} \\[1mm]
E_{6} & \omega_{1},\,\omega_{6} \\[1mm]
E_{7} & \omega_{7}
\end{array}
\end{equation*}

%%%%%%%%%%%%%%%%%%%%%%%%%%%
\subsection{}\label{locfin}
%%%%%%%%%%%%%%%%%%%%%%%%%%%

For any $\lie g$-module $M$ and $\mu\in\lie h^*$, set
\begin{equation*}
M_\mu=\{m \in M \ :\ hm=\mu(h)m,\quad h\in\lie h\}, \qquad
\wt(M)=\left\{\mu \in \lie h^\ast\ :\ M_\mu\ne 0\right\}.
\end{equation*}
We say  $M$ is a \textit{weight module} for $\lie g$ if
$$M=\bigoplus_{\mu\in\lie h^\ast} M_\mu.$$
Any finite-dimensional $\lie g$-module is a weight module.
It is well-known that the set of isomorphism classes of irreducible finite-dimensional
$\lie g$-modules is in bijective correspondence with $P^+$.
For $\lambda\in P^+$ we denote by $V(\lambda)$
a representative of the corresponding isomorphism class.

If $V(\lambda)^*$ is the dual representation of $V(\lambda)$, then \begin{equation}\label{dual} V(\lambda)^*\cong V(-\bow_0\lambda).\end{equation}
The dimension of $V(\lambda)$ is given by the Weyl dimension formula, namely
\begin{equation}\label{weyldim}\dim V(\lambda)=\prod_{\alpha\in R^+}\frac{(\lambda+\rho)(h_\alpha)}{\rho(h_\alpha)},\end{equation}
where $\rho=\sum_{i=1}^n\omega_i\in P^+$. Any finite-dimensional $\lie g$-module is isomorphic to a direct sum of irreducible modules and in particular, we have $$V(\lambda)\otimes V(\mu)\cong\bigoplus_{\nu\in P^+} V(\nu)^{\oplus c^\nu_{\lambda,\mu}},\ \ c^\nu_{\lambda,\mu}=\dim\Hom_{\lie g}(V(\nu), V(\lambda)\otimes V(\mu)).$$
We shall freely use  the fact that $$V(\lambda)\otimes V(\mu)\cong V(\mu)\otimes V(\lambda),$$ and that $$ \dim\Hom_{\lie g}(V(\nu), V(\lambda)\otimes V(\mu))= \dim\Hom_{\lie g}(V(\nu)^*, V(\lambda)^*\otimes V(\mu)^*).$$

%%%%%%%%
\section{The poset \texorpdfstring{$P^+(\lambda,k)/\sim$}{} and the main result}
%%%%%%%%
%
\label{section2}

%%%%%%%%%%%%%
\subsection{}
%%%%%%%%%%%%%
%
Given an integer  $k > 0$ and $ \lambda \in P^+$, set
$$P^+(\lambda,k) =
 \{\blambda= (\lambda_1,\,\dots,\,\lambda_k) \in (P^+)^{\times k} :
   \lambda_1 + \cdots + \lambda_k = \lambda \}.$$
Clearly, $P^+(\lambda,k)$ is a finite set and the symmetric group $S_k$ acts naturally on it.
The Weyl group $W$ acts diagonally on $P^+(\lambda,k)$ as follows:
$$\bow \blambda  = \bow(\lambda_1, \ldots, \lambda_k) = (\bow \lambda_1,\ldots, \bow \lambda_k)
\quad \text{for $\bow \in W$}.
$$
Observe that
\begin{equation} \label{wrelation}
\blambda\in P^+(\lambda,k) \iff
-\bow_0\blambda\in P^+(-\bow_0\lambda,k).
\end{equation}
For $\blambda\in P^+(\lambda,k)$, $\alpha\in R^+$, and $1 \le \ell\le k$,
set
$$r_{\alpha,\ell}(\blambda) :=
\min \{  (\lambda_{i_1} + \cdots + \lambda_{i_\ell})(h_\alpha):
  1 \le i_1 < i_2 < \cdots < i_{\ell}\le k \}.$$
Clearly,
$r_{\alpha,k}(\blambda)\ge r_{\alpha,k-1}(\blambda,\alpha)\ge\cdots\ge r_{\alpha,1}(\blambda)$.
Given $\blambda,\bmu\in P^+(\lambda,k)$, we say that $\blambda\preceq \bmu$ if
$$r_{\alpha,\ell}(\blambda)\le r_{\alpha,\ell}(\bmu)  \quad
 \text{for all $\alpha\in R^+$ and $1\le \ell\le k$}.$$
This defines a preorder on the $P^+(\lambda,k)$.
It induces a partial order on the set of equivalence classes
with respect to the equivalence relation $\sim$
on $P^+(\lambda,k)$, given  by
$$\blambda \sim \bmu \quad \iff \quad
r_{\alpha,\ell}(\blambda)=r_{\alpha,\ell}(\bmu) \quad
\text{for all $\alpha\in R^+$ and  $1\le \ell\le k$}.$$
The poset has a unique minimal element
which is just the $k$-tuple $(\lambda,0,\dots, 0)$.

{\em  For ease of notation, we shall not always distinguish
between elements of $P^+(\lambda,k)/\sim$ and
their representatives in $P^+(\lambda,k)$. }

Note that if $\blambda,\bmu\in P^+(\lambda,k)$, then
\begin{equation} \label{w0}
\blambda\preceq \bmu \quad \iff \quad
-\bow_0\blambda\preceq-\bow_0\bmu.
\end{equation}

We now state the main result of this paper.

\begin{thm}\label{mainthm}
Let $\lie g$ be a finite-dimensional simple Lie algebra,
$k\in\bz_+$ and $\lambda\in P^+$.
Assume that $\blambda =(\lambda_1,\,\dots,\,\lambda_k)$ and
$\bmu=(\mu_1,\,\dots,\,\mu_k)$ are elements of
$P^+(\lambda,k)$ such that $\blambda\preceq\bmu$ in $P^+(\lambda,k)/\sim$.
\begin{enumerit}
\item[(i)]  We have $$\dim\bigotimes_{s=1}^k V(\lambda_s)\le\dim \bigotimes_{s=1}^k V(\mu_s),$$ with equality iff $\blambda=\bmu$ in $P^+(\lambda,k)/\sim$.
\item[(ii)] Let $i\in I$ be such that $\omega_i$ is minuscule and let $\lambda=N\omega_i$ for some $N\in\bz_+$. Then $$\dim\Hom_{\lie g}(V(\nu), \bigotimes_{s=1}^k V(\lambda_s))\le \dim\Hom_{\lie g}(V(\nu), \bigotimes_{s=1}^k V(\mu_s)),\ \ \nu\in P^+.$$
\item[(iii)] If $\lie g$ is of type  $A_2$, and $k=2$, then
  $$\blambda\preceq\bmu\implies \dim\Hom_{\lie g}  (V(\nu), V(\lambda_1) \otimes V(\lambda_2)) \le \dim\Hom_{\lie g}(V(\nu), V(\mu_1) \otimes V(\mu_2)  ),$$ for all $\nu\in P^+$.
\end{enumerit}

\end{thm}
The theorem is proved in the rest of this paper. We conclude this section with some comments on the methods of the proof and give some explanations for the restrictions in the theorem. We also give some context to  our result  by relating them to others in the literature.

\subsection{} To prove part (i) of the Theorem, we use the Weyl dimension formula to reduce the proof to the case of $\lie{sl}_2$. Recall, that the cover relation of $\preceq$ on $P^+(\lambda,k)$ are the pairs $\blambda \prec \bmu$, such that their does not exist $\bnu \in P^+(\lambda,k)$ with $\blambda \prec \bnu \prec \bmu$. For $\lie{sl}_2$ we determine this cover relation which  allows us to determine a sufficient condition for the cover relation in $P^+(\lambda,k)$ in general. Part (ii) of the theorem follows from these ideas along with the Littelmann path model. For an arbitrary simple Lie algebra, it seems quite difficult to determine the cover relations even for $k=2$. For $\lie{sl}_3$ and $k=2$, we  give a sufficient condition for one element to cover another in Section~\ref{section5}. It would appear from our conditions,  that the cover relation depends heavily on the combinatorics of the Weyl group and the root system. Part (iii) of the theorem is proved by using  the information on the cover relation together with results of \cite{KN1994}, \cite{N2002}  on realization of crystal bases.

\subsection{} A subject that has been of much interest has been the notion of Schur positivity. Let $\Lambda$ be the ring of symmetric functions. It is well-known, \cite{M1995}, that it has an integral basis given by the Schur functions $s_\chi$ where $\chi$ is a  partition. A  symmetric function is said to be Schur  positive if it can be written as a non-negative integer linear combination of Schur functions.

 Suppose now that $\lie g$ is of type $A_{r-1}$ and $\lambda\in P^+$, say $\lambda=\sum_{i=1}^{r-1}s_i\omega_i$. Define a partition $\chi(\lambda)=(\chi_1\ge\cdots\ge\chi_r)$ by  $\chi_j=\sum_{p=j}^{r-1} s_p$, if $1\le j\le r-1$ and set $\chi_r=0$.  On the other hand given a partition $\chi = (\chi_1 \geq \chi_2 \geq \ldots \geq \chi_{r-1} \geq 0)$, set $\lambda(\chi)=\sum_{i=1}^{r-1} (\chi_i - \chi_{i+1}) \omega_i\in P^+$.

It is known that the character of $V(\lambda)$ is  $s_{\chi(\lambda)}$.
Parts (ii) and (iii) of Theorem~\ref{mainthm} can be reformulated as: let $\blambda = (\lambda_1, \lambda_2), \bmu = (\mu_1, \mu_2) \in P^+(\lambda,2)$ then $$\blambda \preceq \bmu \implies s_{\chi(\mu_1)} s_{\chi(\mu_2)} - s_{\chi(\lambda_1)} s_{\chi(\lambda_2)} = \sum_{\nu \in P^+} c_{\nu} s_{\chi(\nu)},\ \  c_{\nu} \geq 0\ \ {\rm{for\ all}}\  \nu \in P^+. $$
We conjecture that this is true more generally:
\begin{conj}\label{main-conj}
Let $\lie g$ be a simple Lie algebra, $\lambda \in P^+$, $\blambda = (\lambda_1, \lambda_2), \bmu = (\mu_1, \mu_2) \in P^+(\lambda, 2)$, then
\begin{equation}\label{eqn-dim}
\blambda\preceq\bmu\implies \dim\Hom_{\lie g}  (V(\nu), V(\lambda_1) \otimes V(\lambda_2)) \le \dim\Hom_{\lie g}(V(\nu), V(\mu_1) \otimes V(\mu_2)  ),
\end{equation}
for all $\nu \in P^+$.
\end{conj}
Part (i) of the theorem could be viewed as giving some additional, but very  limited evidence for the conjecture.

\subsection{}  In \cite{FFLP2005} the authors introduced the notion of a row shuffle.
Thus if $\chi = (\chi_1 \geq \chi_2 \geq  \ldots \geq \chi_{n-1} \geq 0)$ and $ \xi = ( \xi_1 \geq  \xi_2 \geq  \ldots \geq  \xi_{n-1} \geq 0)$ are two partitions with at most $n-1$ parts, then the row shuffle is a pair of partitions $(\rho^1(\chi,\xi), \rho^2(\chi,\xi))$ defined  as follows.
Order the multi set $\{ \chi_1, \ldots,
\chi_{r-1}, \xi_1, \ldots,\xi _{n-1} \}$ decreasingly, say $\psi_1 \geq \psi_2 \geq \ldots \geq \psi_{2n-2}$. Set, $$\rho^1(\chi, \xi) = (\psi_1 \geq \psi_3 \geq \ldots \geq \psi_{2n-3} \geq 0),\qquad \rho^2(\chi, \xi) = (\psi_2 \geq \psi_4 \geq \ldots \geq \psi_{2n-2} \geq 0).$$ In other words we are shuffling the rows of the joint partition.
It was conjectured in \cite{FFLP2005} and proved in \cite{LPP2007} that $s_{\rho^1(\chi, \xi)}s_{\rho^2(\chi,\xi)} -  s_\chi s_\xi$ is Schur positive. Related conjectures can also be found in \cite{O1997} and \cite{LLT1997}.  Partial results on this conjecture were also obtained in \cite{BBR2006}, \cite{BM2004}, \cite{Mc2008}, \cite{MW2009}, \cite{PW2008}.

\subsection{} The connection of the row shuffle with our partial order is made as follows. Suppose that $\lie g$ is of type $A_{n-1}$, let $\nu,\mu\in P^+$ and set $\lambda=\mu+\nu$. In Section~\ref{maxslr}, we define an element $\blambda_{\max}=(\lambda^1,\lambda^2)\in P^+(\lambda,2)/\sim$ and prove that it is the unique maximal element in this set. It is a simple calculation to prove that if we take $\xi$ and $\chi$ to be the dual of the partitions $\chi(\nu)$ and $\chi(\mu)$, then $\rho^1(\xi,\chi)$ and $\rho^2(\xi, \chi)$ are the duals of the  partition  $\chi(\lambda^1)$ and $\chi(\lambda^2)$ respectively. It is well-known \cite{M1995} that there exists an involution $\omega$  of the ring $\Lambda$ which maps the Schur function associated to  a partition $\xi$ to the Schur function associated to its dual. Hence the result of \cite{LPP2007} proves the following statement:
let $\mu, \nu \in P^+$ and let $\lambda=\mu+\nu$. Then
$$ \dim\Hom_{\lie g}(V(\eta), V(\mu)\otimes V(\nu))\le \dim\Hom_{\lie g}(V(\eta), V(\lambda^1)\otimes V(\lambda^2)) \text{ for all } \eta \in P^+.
$$
In other words, their result proves that  part (iii) of Theorem \ref{mainthm} is true for $A_{n-1}$ in the case when $\bmu=\blambda_{\max}$.\\
We  remark here, that  in \cite{LPP2007}  the authors proved further,  that in the case where $\lie g$ is of type $A_{n-1}$, having Schur positivity for $k=2$ (and $\bmu = \blambda_{\max}$) implies Schur positivity for arbitrary $k$, e.g. they proved
$$ \dim\Hom_{\lie g}(V(\eta), V(\mu_1 )\otimes \ldots \otimes V(\mu_k))\le \dim\Hom_{\lie g}(V(\eta), V(\lambda^1)\otimes \ldots \otimes V(\lambda^k)) 
$$
for all $k \geq 2$, $(\mu_1, \ldots, \mu_k) \in  \mathcal{P}(\lambda,k)$ and $\eta \in P^+$, where $(\lambda^1, \ldots, \lambda^k)$ is the $k$-row shuffle of $(\mu_1, \ldots, \mu_k)$. One can show (similar to Proposition~\ref{maxslr}), that $(\lambda^1, \ldots, \lambda^k)$ is the unique maximal element in $ \mathcal{P}(\lambda,k)$.

\subsection{}
Suppose $\lie g$ is of type $A_n$, then Lam, Postnikov and Pylyavskyy stated Conjecture~\ref{main-conj} in an unpublished work (we may refer here to \cite{DP2007}). The following first step in proving this conjecture in the $A_{n}$-case has been taken in \cite{DP2007}. It is shown there that for $  (\mu_1, \mu_2) \preceq (\lambda_1, \lambda_2) \in P^+(\lambda, 2)$ and $\nu \in P^+$:
\[
\dim\Hom_{\lie g}(V(\nu), V(\mu_1) \otimes V(\mu_2)  ) \neq 0 \Rightarrow  \dim\Hom_{\lie g}  (V(\nu), V(\lambda_1) \otimes V(\lambda_2)) \neq 0.
\]
Their approach is completely different from ours, as they use the Horn-Klyachko inequalities. It would be interesting to see if some of the ideas apply also when $\lie g$ is not of type $A_n$.

%%%%%%%%%%%%%%%%%%%%%%%%%%%%%%%%%%%%%%%%%%%%%%%%%%%%%%%%%%%%%%%%%%%%%%%%%%%%%%%%%%%%%
%%%%%%%%%%%%%%%%%%%%%%%%%%%%%%%%%%%%% Section 3 %%%%%%%%%%%%%%%%%%%%%%%%%%%%%%%%%%%%%

%%%%%%%%
\section{Proof of \texorpdfstring{Theorem~\ref{mainthm}\,{\rm (i)}}{Theorem 1 (i)}}
%%%%%%%%
%
\label{section3}

The main idea in the proof of part (i) of the theorem is
to show that when $\lambda$ is a multiple of a fundamental weight,
the partial order on $P^+(\lambda,k)$ is determined by
the partial order on $P^+(\lambda,2)$.
We prove this in the first part of the section
and then deduce Theorem~\ref{mainthm}(i).

%%%%%%%%%%%%%
\subsection{}
%%%%%%%%%%%%%

{\em Assume until further notice that $\lie{g}$ is an arbitrary
finite-dimensional simple Lie algebra, and
%
%%%%%%%%%%%%%%
%%% lamome %%%
%%%%%%%%%%%%%%
%
\begin{equation} \label{lamome}
\lambda\in\bz_+ \omega_i \quad \text{for some $i\in I$}.
\end{equation}
In particular, we can and will think of $\lambda$ as a non-negative integer.
It is clear that elements of $P^+(\lambda,k)$ are just $k$-tuples of non-negative integers
which add up to $\lambda \in \bz_{+}$. If $\blambda\in P^+(\lambda,k)$, we have
$$r_{\alpha,\ell}(\blambda)=d_ir_{\alpha_i,\ell}(\blambda),\ \  h_\alpha=\sum_{i\in I}d_ih_i,$$
and it follows that if $\blambda,\bmu\in P^+(\lambda,k)$, then
$$\blambda\preceq\bmu \iff
r_{\alpha_i,\ell}(\blambda) \le r_{\alpha_i,\ell}(\bmu), \quad 1 \le \ell \le k. $$
In other words, the partial order is determined entirely by $\alpha_i$.
So, we shall drop the dependence on $\alpha$ and write $r_\ell(\blambda)$
for $r_{\alpha_i,\ell}(\blambda)$.%
}

%%%%%%%%%%%%%
\subsection{}
%%%%%%%%%%%%%
%
\label{section-classes}

%
%%%%%%%%%%%%%%%%
%%% Sk-orbit %%%
%%%%%%%%%%%%%%%%
%
\begin{lem} \label{Sk-orbit}
Let $\blambda,\,\bmu\in P^+(\lambda,k)$.
Then, $\blambda\sim\bmu$ iff $\bmu=\sigma \blambda$ for some $\sigma\in S_k$,
i.e. the equivalence class of $\blambda$ is exactly the $S_k$ orbit of $\blambda$.
\end{lem}

\begin{pf}
The fact that $\blambda=\sigma \bmu$ implies $\blambda\sim\bmu$ is
clear from the definition of $P^+(\lambda,k)$.
For the converse, choose $\sigma,\sigma'\in S_k$
so that $\sigma\blambda$ and $\sigma'\bmu$ are partitions of $\lambda$, say
$$\sigma\blambda=(\lambda_k\ge\cdots \ge\lambda_1),\qquad
 \sigma'\bmu=(\mu_k\ge\dots\ge\mu_1).$$
Since $\blambda\sim\sigma \blambda$,
it follows that $\sigma\blambda\sim\sigma'\bmu$.
But this implies, that
$$\lambda_1=\mu_1, \quad
  \lambda_1+\lambda_{2}=\mu_1+\mu_{2}, \quad \dots, \quad
  \lambda_1 + \lambda_2+\cdots + \lambda_k=
  \mu_1+\cdots+\mu_k,$$
forcing $\sigma\blambda=\sigma'\bmu$ and the Lemma is proved.
\end{pf}

{\em From now on, we will identify the set $P^+(\lambda, k)/\sim$ of equivalence classes
with partitions of $\lambda$ with at most $k$ parts.
By abuse of notation we  continue to denote this set
as $P^+(\lambda,k)$ and note that $\preceq$ is now a partial order on this set.
As a consequence, we shall also assume without comment that
$r_\ell(\blambda)=\lambda_1+\cdots+\lambda_\ell$ for $1\le \ell\le k$.}

%%%%%%%%%%%%%
\subsection{}
%%%%%%%%%%%%%
%
\label{maxelement}

Let $\lambda$ be as in \eqref{lamome}, and $k\in\bz_+$.
Write $\lambda=k\lambda_0+p_0$ where $0\le p_0<k$ and $\lambda_0\in\bz_+$.
Define $\blambda_{\max}=(\lambda_k\ge\cdots\ge \lambda_1)\in P^+(\lambda,k)$ by
$$\lambda_j=
 \begin{cases}
   \lambda_0 & 1 \le j \le k - p_0, \\
   \lambda_0+1 & j > k - p_0.
 \end{cases}$$
Observe that
$$r_\ell(\blambda_{\max})=
 \begin{cases}
   \ell \lambda_0 & \text{if $1\le  \ell \leq k-p_0$}, \\
   (\ell+1) \lambda_0 -k+p_0 & \text{otherwise}.
 \end{cases}.$$
The following result justifies the notation.

\begin{lem}\label{maxsl2lem}
Keep the notation above.
For all $\blambda\in P^+(\lambda,k)$, we have $\blambda\preceq \blambda_{\max}$.
Moreover, $\blambda_{\max}$ is the unique element of $P^+(\lambda,k)$ satisfying
\begin{equation}\label{minmax}
\max_{1 \leq i \leq k} \{\lambda_i\} - \min_{1 \leq i \leq k}\{\lambda_i\} \leq 1.
\end{equation}
\end{lem}

\begin{proof}
It is clear that $\blambda_{\max}$ satisfies \eqref{minmax}.
If $\bmu=(\mu_k \geq \cdots \geq \mu_1) \in P^+(\lambda,k)$ also
satisfies \eqref{minmax}, then $\mu_j-\mu_1\le 1$ for all $1\le j\le k$.
If $\mu_j=\mu_1$ for all $j$ then $\lambda=k\mu_1$ and hence
we would have $p_0=0$ and $\mu_1=\lambda_0$ as required.
Otherwise, fix $j_0\le k-1 $ such that $\mu_j=\mu_1$ for
$1\le j\le k-j_0$ and $\mu_j=\mu_1+1$ otherwise. Then
$$\sum_{s=1}^k\mu_s=k\mu_1+ j_0 =\lambda= k\lambda_0+k_0=
  \sum_{s=1}^k\lambda_s,$$
and hence $\mu_1=\lambda_0$ and $j_0=k_0$,
which proves that $\bmu=\blambda_{\max}$.

Suppose for a contradiction that there exists
$\blambda =  (\lambda_k \ge \cdots \ge \lambda_1) \in P^+(\lambda,k)$ and
$1\le j\le k-p_0$ with $r_j(\blambda)>j\lambda_0$ and
assume that $j$ is minimal with this property.
Since $$r_j(\blambda)=\lambda_j+r_{j-1}(\blambda)>j\lambda_0,$$ we get  $\lambda_\ell\ge\lambda_j\ge  \lambda_0+1$ for all $\ell\ge j$. This gives,
$$\lambda=\lambda_k+\cdots+\lambda_{j+1}+r_j(\blambda)
>(k-j)(\lambda_0+1)+j\lambda_0\ge k\lambda_0+p_0 = \lambda,$$ which is a contradiction. The case $j>k-p_0$ is similar and we omit the details.
\end{proof}

%%%%%%%%%%%%%
\subsection{}
%%%%%%%%%%%%%
%
We now  prove that the partial order on $P^+(\lambda,k)$ is entirely determined by the partial order $P^+(\lambda,2)$. The first step is the following result which determines the cover relation in $P^+(\lambda,2)$.
\begin{lem}\label{cover}
Suppose that $\blambda=(\lambda_1\ge\lambda_2)\in P^+(\lambda,2)$ and
assume that $\blambda\ne\blambda_{\max}$.
Then $\bmu \in P^+(\lambda,2)$ covers $\blambda$ iff
$$\bmu =(\lambda_2-1\ge\lambda_1+1).$$
\end{lem}

\begin{proof}
Since  $\blambda\ne\blambda_{\max}$, we see from Lemma \ref{maxsl2lem} that
$\lambda_2-\lambda_1>1$. Hence, $\bmu = (\lambda_2-1\ge\lambda_1+1)\in P^+(\lambda,k)$ and
$\blambda \preceq \bmu$. Suppose that there exists $\bnu=(\nu_2\ge\nu_1)$
with $\blambda \prec \bnu$. Then $\lambda_1<\nu_1$ and hence
we get $\bmu \preceq \bnu$ which proves the lemma.
\end{proof}

%%%%%%%%%%%%%
\subsection{}
%%%%%%%%%%%%%
%
\label{2impliesk}

Let $\blambda=(\lambda_k \ge \cdots \ge \lambda_1) \in P^+(\lambda,k)$
be such that $\lambda_k-\lambda_1\ge 2$. Then there exists
$1 \le j_1 < j_2\le k$ with
$$\lambda_{j_1} < \lambda_{j_1+1}, \qquad
  \lambda_{j_2-1}<\lambda_{j_2}, \qquad
  \lambda_{j_1} +2 \leq \lambda_{j_2}.$$
For each such pair of $(j_2,j_1)$, we define a partition
$\blambda(j_2,j_1)=(\lambda_k' \ge \cdots \ge \lambda_1')$ by:
\begin{equation}\label{shift2}
\lambda_i'=
\begin{cases}
 \lambda_i & i \ne j_1,j_2,\\
 \lambda_{j_2}-1 & i=j_2,\\
 \lambda_{j_1}+1 & i=j_1.
\end{cases}
\end{equation}
Observe that
\begin{equation} \label{shift}
r_\ell(\blambda(j_2,j_1))=
 \begin{cases}
   r_\ell(\blambda)+1 & j_1\le \ell<j_2, \\
   r_\ell(\blambda) & \text{otherwise},
 \end{cases}
\end{equation}
and so $\blambda \prec \blambda(j_2,j_1)$.
The following proposition shows that the partial order on
$P^+(\lambda,k)$ is controlled by the partial order on $P^+(\lambda,2)$.
%
%%%%%%%%%%%%%%%%%%%
%%% prop:cover1 %%%
%%%%%%%%%%%%%%%%%%%
%
\begin{prop} \label{prop:cover1}
Let $\blambda=(\lambda_k \ge \cdots \ge \lambda_1)\in P^+(\lambda, k)$ and
assume that $\bmu\in P^+(\lambda,k)$ covers $\blambda$.
Then $\bmu=\blambda(j_2,j_1)$ for some $1\le j_1<j_2\le k$
with $\lambda_{j_1} < \lambda_{j_1 + 1}$, $\lambda_{j_2-1}<\lambda_{j_2}$ and
$\lambda_{j_1} +2 \leq \lambda_{j_2}$.
\end{prop}

\begin{pf}
We proceed by induction on $k$. Lemma \ref{cover} shows that induction begins at $k=2$.
Assume that $k > 2$.
%
% For the inductive step, assume the result for all $\bnu\in P^+(\nu,s)$
% for all $\nu\in P^+$ and $s<k$.
%
Let $\blambda\in P^+(\lambda,k)$, and
assume that $\bmu\in P^+(\lambda,k)$ covers $\blambda$.

Suppose $r_\ell(\blambda)<r_\ell(\bmu)$ for all $1\le \ell < k$.
Since $\blambda$ has a cover, it follows that $\blambda\ne\blambda_{\max}$
and so there exists $1\le j_1<j_2\le \ell$ such that $\blambda(j_2,j_1)$ is defined.
Now, \eqref{shift} shows that $$\blambda \prec  \blambda(j_2,j_1)\preceq\bmu,$$
and hence $\bmu=\blambda(j_2,j_1)$.

Suppose now that there exists $1\le \ell<k$ such that
$r_\ell(\bmu)=r_\ell(\blambda)$ and $\ell$ is minimal with this property.
Consider first the case when $\ell=1$, i.e., $\mu_1=\lambda_1$. Then,
$$\bmu_0=(\mu_k\ge\cdots \ge\mu_2), \qquad
  \blambda_0=(\lambda_k\ge\cdots\ge \lambda_2),$$
are distinct elements of $P^+(\lambda-\lambda_1, k-1)$ (since $\bmu\ne\blambda$).
Moreover, we claim that $\bmu_0$ covers $\blambda_0$.
If there exists $\bnu_0=(\nu_k\ge\cdots\ge \nu_2)\in P^+(\lambda-\lambda_1, k-1)$
such that
$$\blambda_0 \prec \bnu_0 \preceq \bmu_0,$$
then $\nu_2\ge\lambda_2\ge\lambda_1$. Hence, if we set
$\bnu=(\nu_k \geq \cdots \geq \nu_2 \geq \lambda_1) \in P^{+}(\lambda,\,k)$,
then we get
$$\blambda\prec\bnu\preceq\bmu.$$
This forces $\bnu=\bmu$ and hence $\bnu_0=\bmu_0$.
By induction on $k$ and noting that $k-1\ge 2$, we see that
$$\bmu_0=\blambda_0(j_2,j_1),$$
for some $2\le j_1<j_2\le k$ and hence $\bmu=\blambda(j_2,j_1)$.

It remains to consider the case when $\ell\ge 2$,
in particular, this means that $\mu_1>\lambda_1$.
This time, we take
$$\bmu_0=(\mu_\ell\ge\cdots\ge\mu_1), \qquad
  \blambda_0=(\lambda_\ell\ge\cdots\ge\lambda_1),$$
and note that these are elements of $P^+(r_\ell(\blambda), \ell)$ and
that $\blambda_0\prec \bmu_0$.
We claim again that $\bmu_0$ covers $\blambda_0$.
Thus, let $\bnu_0=(\nu_\ell\ge\cdots\ge\nu_1) \in P^+(r_\ell(\blambda), \ell)$
be such that
$$\blambda_0\prec\bnu_0\preceq\bmu_0.$$
Then
$$\nu_s+\cdots+\nu_1 \ge \lambda_s+\cdots+\lambda_1, \quad 1 \le s\le \ell, \qquad
  \nu_\ell+\cdots+\nu_1 = \lambda_\ell+\cdots+\lambda_1. $$
Suppose that $\nu_\ell>\lambda_{\ell+1}$.
Then we get $\nu_\ell > \lambda_{\ell+1} \ge \lambda_{\ell}$, and
$$0<\nu_\ell-\lambda_\ell = (\lambda_{\ell-1}+\cdots+\lambda_1)-(\nu_{\ell-1}+\cdots+\nu_1)\le 0,$$
which is a contradiction. Thus we get $\nu_\ell \le \lambda_{\ell+1}$, and hence
$$\bnu=(\lambda_k \ge \cdots \ge \lambda_{\ell+1} \ge \nu_{\ell} \ge \cdots \ge \nu_{1})
  \in P^+(\lambda,k).$$
Also we see that $\blambda \prec \bnu\preceq\bmu$. Since $\bmu$ is a cover of
$\blambda$, this forces $\bnu =\bmu$, and hence $\bnu_{0}=\bmu_0$.
Thus we conclude that $\bmu_0$ covers $\blambda_0$. By induction on $k$,
$\bmu_0=\blambda_0(j_2,j_1)$ for some $1\le j_1<j_2\le \ell$.
We see that $\blambda \prec \blambda(j_2,j_1) \preceq \bmu$,
which forces $\bmu=\blambda(j_2,j_1)$. Thus we have proved the proposition.
\end{pf}

%%%%%%%%%%%%%
\subsection{}
%%%%%%%%%%%%%
%
\label{sl2}

{\em Proof of  Theorem~\ref{mainthm} (i).}\
Assume first that  $\lie g$ is $\lie{sl}_{2}$.
Let $\lambda \in P^{+}=\bz_{+}\omega_{1}$ be an arbitrary
dominant integral weight, and $k \ge 2$.
Let $\blambda=(\lambda_k\ge\cdots\ge\lambda_1) \in P^{+}(\lambda,k)$ and
$\bmu=(\mu_k\ge\cdots\ge\mu_1) \in P^{+}(\lambda,k)$ be such that
$\blambda \preceq \bmu$ (in $P^{+}(\lambda,k)/\!\sim$).
First we show that if $\blambda \prec \bmu$ (in $P^{+}(\lambda,k)/\!\sim$),
then
%
%%%%%%%%%%%%
%%% dim1 %%%
%%%%%%%%%%%%
%
\begin{equation} \label{dim1}
\dim \bigotimes_{s=1}^k V(\lambda_s) <
\dim \bigotimes_{s=1}^k V(\mu_s), \quad \text{i.e.,} \quad
\prod_{s=1}^{k}(\lambda_{s}+1) < \prod_{s=1}^{k}(\mu_{s}+1).
\end{equation}
A standard argument shows  that
there exists a sequence $\blambda=\bnu_{0},\,\bnu_{1},\,\dots,\,\bnu_{p}=\bmu$
of elements of $P^{+}(\lambda,k)$ such that $\bnu_{q}$ covers $\bnu_{q-1}$
for each $1 \le q \le p$. It suffices to show \eqref{dim1}
in the case when $\bmu$ covers $\blambda$.
Then, by Proposition~\ref{prop:cover1}, there exists $1 \le j_{1} < j_{2} \le k$
with $\lambda_{j_1} < \lambda_{j_1+1}$, $\lambda_{j_2-1} < \lambda_{j_2}$, and
$\lambda_{j_1}+2 \le \lambda_{j_2}$
such that $\bmu=\blambda(j_{2},j_{1})$.
Thus the inequality \eqref{dim1} is equivalent to
\begin{equation*}
(\lambda_{j_{1}}+1)(\lambda_{j_{2}}+1) <
(\lambda_{j_{1}}+2)\lambda_{j_{2}}.
\end{equation*}
But this is obvious from the fact that
$\lambda_{j_1}+2 \le \lambda_{j_2}$.
Thus we have proved \eqref{dim1}.
Also, we have proved that
(under the assumption that $\blambda \preceq \bmu$) if
\begin{equation*}
\dim \bigotimes_{s=1}^k V(\lambda_s) =
\dim \bigotimes_{s=1}^k V(\mu_s),
\end{equation*}
then $\blambda = \bmu$ (in $P^{+}(\lambda,k)/\!\sim$).
The converse of this statement is obvious by Lemma~\ref{Sk-orbit}.
Thus we have proved Theorem~\ref{mainthm}\,(i) in the case of
$\lie{sl}_{2}$.

Assume next that $\lie g$ is an arbitrary finite-dimensional simple complex Lie algebra,
and $\lambda \in P^+$ is an arbitrary dominant integral weight.
Let $\blambda=(\lambda_{1},\,\dots,\,\lambda_{k}),\,
\bmu=(\mu_{1},\,\dots,\,\mu_{k}) \in P^+(\lambda,k)$ be such that
$\blambda \preceq \bmu$ (in $P^+(\lambda,k)/\!\sim$).
Using the Weyl dimension formula we see that
$$\dim\bigotimes_{s=1}^k V(\lambda_s)=
 \prod_{s=1}^k\prod_{\alpha\in R^+}
 \frac{(\lambda_s+\rho)(h_\alpha)}{\rho(h_\alpha)}, \qquad
 \dim\bigotimes_{s=1}^k V(\mu_s)=
 \prod_{s=1}^k\prod_{\alpha\in R^+}
 \frac{(\mu_s+\rho)(h_\alpha)}{\rho(h_\alpha)}.$$
So, in order to prove that
$\dim\bigotimes_{s=1}^k V(\lambda_s) \le
 \dim\bigotimes_{s=1}^k V(\mu_s)$, it suffices to show that
\begin{equation*}
\prod_{s=1}^{k} (\lambda_s+\rho)(h_\alpha) \le
\prod_{s=1}^{k} (\mu_s+\rho)(h_\alpha)
\qquad \text{for each $\alpha \in R^{+}$}.
\end{equation*}
For each $\alpha \in R^{+}$ and $1 \le s \le k$, we set
\begin{equation*}
\lambda_{s}^{(\alpha)}=
 (\lambda_s + \rho) (h_\alpha) - 1, \qquad
\mu_{s}^{(\alpha)}=
 (\mu_s + \rho) (h_\alpha) - 1,
\end{equation*}
\begin{equation*}
\lambda^{(\alpha)}= \lambda(h_\alpha) + k(\rho(h_\alpha) - 1) =
\sum_{s=1}^{k} \lambda_{s}^{(\alpha)} =
\sum_{s=1}^{k} \mu_{s}^{(\alpha)}.
\end{equation*}
Then we see that the elements
\begin{equation*}
\blambda^{(\alpha)}=
(\lambda_{1}^{(\alpha)},\,\dots,\,\lambda_{k}^{(\alpha)}), \qquad
\mu^{(\alpha)}=
(\mu_{1}^{(\alpha)},\,\dots,\,\mu_{k}^{(\alpha)})
\end{equation*}
are elements of $P^{+}( \lambda^{(\alpha)},k)$ for $\lie{sl}_2$
(or rather $\lie{sl}_2(\alpha)$) satisfying
$\blambda^{(\alpha)} \preceq \bmu^{(\alpha)}$
(in $P^{+}( \lambda^{(\alpha)},k)/\!\sim$).
Hence, by the argument for $\lie{sl}_{2}$ above, we obtain
\begin{equation*}
\prod_{s=1}^{k} (\lambda_s+\rho)(h_\alpha) =
\prod_{s=1}^{k} (\lambda_s^{(\alpha)}+1) \le
\prod_{s=1}^{k} (\mu_s^{(\alpha)}+1) =
\prod_{s=1}^{k} (\mu_s+\rho)(h_\alpha),
\end{equation*}
as desired. Also, we deduce that
(under the assumption that $\blambda \preceq \bmu$)
\begin{align*}
&
\dim\bigotimes_{s=1}^k V(\lambda_s) =
 \dim\bigotimes_{s=1}^k V(\mu_s) \\
& \qquad \iff \quad
\prod_{s=1}^{k} (\lambda_s^{(\alpha)}+1) =
\prod_{s=1}^{k} (\mu_s^{(\alpha)}+1) \quad
\text{for all $\alpha \in R^{+}$} \\[1.5mm]
& \qquad \iff \quad
\blambda^{(\alpha)} = \bmu^{(\alpha)} \quad
\text{for all $\alpha \in R^{+}$ (by the argument for $\lie{sl}_{2}$ above)} \\
& \qquad \iff \quad
\blambda = \bmu.
\end{align*}
Thus we have proved Theorem~\ref{mainthm}\,(i)  \qed

%%%%%%%%
\section{Proof of \texorpdfstring{Theorem~\ref{mainthm}\,{\rm (ii)}}{Theorem 1 (ii)}} \label{section4} %Section 4
%%%%%%%%

%%%%%%%%%%%%%
\subsection{} % 4.1
%%%%%%%%%%%%%
%
As in Section~\ref{section3}, we regard elements of
$P^+(N\omega_i,k)$ as partitions of $N$.
Also, we deduce from Proposition~\ref{prop:cover1} that
Theorem \ref{mainthm}\,(ii) is proved
once we establish the following proposition.
%
%%%%%%%%%%%
%%% min %%%
%%%%%%%%%%%
%
\begin{prop}\label{min}
Suppose that $r,s\in\bz_+$ and assume that $s\ge r+1$.
Let $i\in I$ be such that $\omega_i$ is minuscule.
Then, for all $\mu \in P^+$, we have
\begin{equation*}
\dim \Hom_{\lie g}(V(\mu),\,V(s\omega_i) \otimes V(r\omega_i)) \le
\dim\Hom_{\lie g}(V(\mu),\,V((s-1)\omega_i)\otimes V((r+1)\omega_i)).
\end{equation*}
\end{prop}

The proposition is established in the rest of
the section using the Littelmann path model.

%%%%%%%%%%%%%
\subsection{}
%%%%%%%%%%%%%
%
We recall the essential definitions and results from \cite{L1, L2}.

\begin{defn}
\begin{enumerit}
\item[(i)]
Let $\lambda\in P^+$ and  $\mu,\,\nu \in W\lambda$. We say that
$\mu \ge \nu$ if there exists a sequence
$\mu=\xi_{0},\,\xi_{1},\,\dots,\,\xi_{m}=\nu$ of elements
in $W\lambda$ and elements
$\beta_{1},\,\dots,\,\beta_{m} \in R^{+}$ of positive roots
such that $$\xi_{p}=\bos_{\beta_{p}}(\xi_{p-1}),\ \
\xi_{p-1}(h_{\beta_p}) < 0,\ \  1\le p\le m.$$
 Moreover, in this case, we let $\dist(\mu,\nu)$
be the maximal length $m$ of all  such possible sequences.

\item[(ii)]
For $\mu,\,\nu \in W\lambda$ with $\mu > \nu$ and
a rational number $0 < a < 1$, we define
an $a$-chain for $(\mu,\nu)$ as  a sequence $\mu=\xi_{0} > \xi_{1} >
\dots > \xi_{m}=\nu$ of elements in $W\lambda$
such that $$\dist(\xi_{p-1},\xi_{p})=1,\ \  \xi_{p}=\bos_{\beta_{p}}(\xi_{p-1}),\ \
a\xi_{p-1}(h_{\beta_{p}}) \in \bz_{< 0},$$
for all $p=1,\,2,\,\dots,\,m$.

\item[(iii)]
An LS path of shape $\lambda$ is
a pair $(\ud{\nu}\,;\,\ud{a})$ consisting of a  sequence
$\ud{\nu}=(\nu_{1}>\,\nu_{2}>\dots > \nu_{\ell})$ (for some $\ell\ge 1$) of elements
in $W\lambda$ and a sequence
$\ud{a}=(0=a_{0} < a_{1} < \cdots < a_{\ell}=1)$
of rational numbers satisfying the condition that
there exists an $a_{p}$-chain for $(\nu_{p},\,\nu_{p+1})$
for $p=1,\,2,\,\dots,\,\ell-1$.
We denote by $\BB(\lambda)$ the set of all LS paths of shape $\lambda$.\end{enumerit}
\end{defn}

%%%%%%%%%%%%%
\subsection{}
%%%%%%%%%%%%%

Set $\lie{h}^{\ast}_{\br}=\sum_{i=1}^{n}\br\omega_{i}$, where $\br$ is the set of real numbers.
Given an LS path
$\pi=(\ud{\nu}\,;\,\ud{a})=(\nu_{1},\,\nu_{2},\,\dots,\,\nu_{\ell}\,;\,
a_{0},\,a_{1},\,\dots,\,a_{\ell})$ of shape $\lambda$, define a  piecewise linear,
continuous map $\pi:[0,1] \rightarrow \lie{h}^{\ast}_{\br}$ by:
\begin{equation} \label{eq:path}
\pi(t)=\sum_{p=1}^{q-1}
(a_{p}-a_{p-1})\nu_{p}+
(t-a_{q-1})\nu_{q} \quad \text{for \ }
a_{q-1} \le t \le a_{q}, \  1 \le q\le \ell.
\end{equation}
Clearly distinct LS paths give rise to distinct
piece-wise  linear functions
with values in {$\lie{h}^{\ast}_{\br}$} and we shall make
this identification freely in what follows.

Given $\xi \in P^+$, we say that an LS path $\pi$ of shape $\lambda$ is
$\xi$-dominant if $(\xi+\pi(t))(h_{i}) \ge 0$ for all $i \in I$ and $t \in [0,1]$.
Note that ${\pi=
(\nu_{1},\,\nu_{2},\,\dots,\,\nu_{\ell}\,;\,
a_{0},\,a_{1},\,\dots,\,a_{\ell})}$ is $\xi$-dominant
if and only if $(\xi+\pi(a_{p}))(h_{i}) \ge 0$
for all $i \in I$ and $0 \le p \le \ell$.

For {$\lambda,\,\xi,\,\mu \in P^+$}, set
\begin{equation}
\BB(\lambda)^{\mu}_{\dom{\xi}}=
 \bigl\{\pi\in\BB(\lambda):
 \text{$\pi$ is $\xi$-dominant, and $\xi+\pi(1)=\mu$}
 \bigr\}.
\end{equation}

The following was proved in  \cite{L1}.

\begin{thm*} \label{thm:LR}
For $\lambda,\xi,\mu \in P^+$, we have

\begin{equation}
\dim\Hom_{\lie g}(V(\mu),\,V(\xi) \otimes V(\lambda))=
 \#\BB(\lambda)^{\mu}_{\dom{\xi}}.
\end{equation}
\end{thm*}
%

%%%%%%%%%%%%%
\subsection{}
%%%%%%%%%%%%%

The first step in the proof of Proposition \ref{min} is
to describe the set $\BB(N\omega_i)$ explicitly
when $\omega_i$ is minuscule.
\begin{lem} \label{lem:LS}
Let $i\in I$ be such that $\omega_{i}$ is minuscule.
Consider a pair $(\ud{\nu}\,;\,\ud{a})$,
where $\ud{\nu}=(\nu_{1}>\,\nu_{2}>\dots > \nu_{\ell})$ is
a sequence of elements in $W (N\omega_i)$ and
$\ud{a}=(0=a_{0} < a_{1} < \cdots < a_{\ell}=1)$ is
a sequence of rational numbers (for some $\ell\ge 1$).
Then we have:
$$
 (\ud{\nu}\,;\,\ud{a}) \in \BB(N\omega_i) \iff Na_p \in \bz_+ \quad
 \text{\rm {for all $0 \le p \le \ell$} }.
$$
\end{lem}

\begin{proof}
Suppose first that $(\ud{\nu},\ud{a})$ is such that
$Na_p\in\bz_+$ for all $0\le p\le \ell$,
in which case we must prove that for $1\le p \le {\ell-1}$,
there exists an $a_p$-chain {for} $(\nu_{p},\,\nu_{p+1})$.
Since $\nu_{p} > \nu_{p+1}$,
there exists a  sequence ${\nu_{p}}=\xi_{0} > \xi_{1} >
\dots > \xi_{m}={\nu_{p+1}}$ of elements in $W(N\omega_{i})$
such that
$$\dist(\xi_{{q-1}},\xi_{{q}})=1,\ \
  \xi_{{q}}=\bos_{\beta_{{q}}}(\xi_{{q-1}}),\ \
  \xi_{{q}-1}(h_{\beta_{q}}) < 0,\ \  1 \le {q} \le m.$$
Writing $\xi_{p-1}=\bow(N\omega_{i})$ with some $\bow \in W$, we get
\begin{equation*}
\xi_{p-1}(h_{\beta_{p}})=N\underbrace{(\bow\omega_{i})(h_{\beta_{p}})}_{\in \bz} \in N\bz_{<0},
\end{equation*}
which gives
$$a_p\xi_{p-1}(h_{\beta_{p}})\in\bz,$$
as required.

We now suppose that
 $(\ud{\nu},\ud{a})=(\nu_{1},\,\nu_{2},\,\dots,\,\nu_{\ell}\,;\,
a_{0},\,a_{1},\,\dots,\,a_{\ell}) \in \BB(N\omega_{i})$.
Observe that, if $\nu\in W(N\omega_i)$, then
since $\omega_i$ is minuscule,
we have that $\nu(h_\beta)\in\{0,\pm N\}$ for all $\beta\in R$.
We have to prove that  $Na_p\in\bz_+$ for $1\le p\le \ell$.
The assertion is obvious when $p=0$ or $\ell$. If $1 \le p \le \ell-1$,
{choose} an $a_p$-chain, $\nu_{p}=\xi_{0} > \xi_{1} > \dots > \xi_{m}=\nu_{p+1}$.
{Then} there exists a positive root $\beta$ such that
$$
 \xi_{1}=\bos_{\beta}(\xi_{0}), \ \
 a_{p}\xi_{0}(h_{\beta}) {\in \bz_{< 0}}.
$$
In particular, {we have $\xi_{0}(h_{\beta}) < 0$, which implies that $\xi_{0}(h_{\beta})=-N$.
Thus we get $Na_p\in\bz_+$ as required.}

\end{proof}

%%%%%%%%%%%%%
\subsection{}
%%%%%%%%%%%%%
%
The following observations are  trivial but useful:
{%
%
%%%%%%%%%%%%
%%% eq:r %%%
%%%%%%%%%%%%
%
\begin{equation} \label{eq:r}
\begin{cases}
\nu \in W(r\omega_i) \implies \nu'= \frac{r+1}{r}\nu \in W((r+1)\omega_i), \\[1.5mm]
\nu,\,\gamma \in W(r\omega_i),\,\nu > \gamma \implies
  \nu'> \gamma', \\[1.5mm]
0\le a<1\implies \ 0 \le a'=\frac{ra}{r+1}<1, \quad (r+1)a'\in\bz_+.
\end{cases}
\end{equation}
}

Given $(\ud{\nu};\ud{a})=
(\nu_1,\,\dots,\,\nu_\ell\,;\,a_0,\,\dots,\,a_\ell)$, set
$$(\ud{\nu}'\,;\,\ud{a}')=
\begin{cases}
\left(\nu_1',\,\dots,\,\nu_\ell',\,(r+1)\omega_i\,;\,
  a_0',\,\dots,\,a_\ell',\,1\right)
& { \text{if $\nu_\ell \ne r\omega_i$} }, \\[1.5mm]
\left(\nu_1',\,\dots,\,\nu_{\ell-1}',\,(r+1)\omega_i\,;\,
  a_0',\,\dots,\,a_{\ell-1}',\,1\right)
& { \text{if $\nu_\ell=r\omega_i$} }.
\end{cases}
$$

%%%%%%%%%%%%%

We now prove  Proposition~\ref{min} by showing that
for each $\mu \in P^{+}$, the assignment
$$ (\ud{\nu};\ud{a})\to (\ud{\nu'};\ud{a'})$$
gives an injective map
$$\iota_r:
 \BB(r\omega_{i})_{\dom{s\omega_{i}}}^{\mu} \hookrightarrow
 \BB((r+1)\omega_{i})_{\dom{(s-1)\omega_{i}}}^{\mu}.
$$
It is immediate from Lemma~\ref{lem:LS},{along with \eqref{eq:r},} and
the fact that $(r+1)\omega_{i}$ is the minimum element in $W((r+1)\omega_{i})$
(with respect to the ordering $>$ (see also \cite[Remark 4.2]{L2})) that
$$(\ud{\nu}\,;\,\ud{a}) \in {\BB(r\omega_{i})}
  \implies
  (\ud{\nu'}\,;\,\ud{a'}) \in {\BB((r+1)\omega_{i})}.
$$
Let $\pi$ and $\pi'$ be the piecewise linear paths associated to
$(\ud{\nu}\,;\,\ud{a})$ and $(\ud{\nu'}\,;\,\ud{a'})$, respectively
(see \eqref{eq:path}). We have
\begin{equation} \label{eq:pi2}
\pi'(t)=
\begin{cases}
 \pi\left(\dfrac{r+1}{r}t\right) & \text{for } 0 \le t \le \dfrac{r}{r+1} \\[5mm]
 \pi(1)+\left(t-\dfrac{r}{r+1}\right)(r+1)\omega_{i}
 &  \text{for } \dfrac{r}{r+1} \le t \le 1.
\end{cases}
\end{equation}
This proves immediately that
$$s\omega_{i}+\pi(1)={\mu} \implies
(s-1)\omega_{i}+\pi'(1)=(s-1)\omega_{i}+\pi(1)+\omega_{i}=
s\omega_{i}+\pi(1)={\mu}.$$
Moreover, since
$$t \in \left[0,\,\frac{r}{r+1}\right] \iff
\frac{r+1}{r} t\in [0,\,1],$$
it follows also that
{%
if $\eta$ corresponds to an element of $\BB(s\omega_{i})$
different from $(\ud{\nu}\,;\,\ud{a})$, then
there exists $t \in \left[0,\frac{r}{r+1}\right]$ such that
$$\pi'(t) \ne \eta'(t).$$
}%
Thus we have proved that $\iota_r$ is injective.

It  remains to show that $\pi'$ is $(s-1)\omega_{i}$-dominant.
Let $j \in I$. If $j \ne i$, then we have
\begin{equation*}
((s-1)\omega_{i}+\pi'(t))(h_{j})=(\pi'(t))(h_{j}).
\end{equation*}
Since $\pi$ is $s\omega_{i}$-dominant, we have
\begin{equation*}
0 \le (s\omega_{i}+\pi(t))(h_{j})=(\pi(t))(h_{j})
  \quad \text{for all $0 \le t \le 1$}.
\end{equation*}
Thus, by \eqref{eq:pi2}, we see that
$(\pi'(t))(h_{j}) \ge 0$  for all $0 \le t \le \frac{r}{r+1}$.
Also, for $\frac{r}{r+1} \le t \le 1$, we have
\begin{equation*}
(\pi'(t))(h_{j}) = (\pi(1))(h_{j}) +
 \left(t-\dfrac{r}{r+1}\right)(r+1)\underbrace{\omega_{i}(h_{j})}_{=0} =
 (\pi(1))(h_{j}) \ge 0.
\end{equation*}
Thus we have shown that
if $j \ne i$, then $((s-1)\omega_{i}+\pi'(t))(h_{j}) \ge 0$ for all
$0 \le t \le 1$.

Next, assume that $j=i$.
We see from \eqref{eq:pi2} that
the function $((s-1)\omega_{i}+\pi'(t))(h_{i})$ is strictly
increasing on $\left[\frac{r}{r+1},\,1\right]$.
Thus it suffices to show that
%
%%%%%%%%%%%%
%%% eq:i %%%
%%%%%%%%%%%%
%
\begin{equation} \label{eq:i}
((s-1)\omega_{i}+\pi'(t))(h_{i}) \ge 0
\quad \text{for all } 0 \le t \le \frac{r}{r+1}.
\end{equation}
Let $0 \le q \le \ell$. We have
\begin{align*}
(\pi'(a_{q}'))(h_{i})=
\sum_{p=1}^{q} (a_{p}'-a_{p-1}')\nu_{p}'(h_{i}).
\end{align*}
Here, we note that $\nu_{p}'(h_{i}) \in \bigl\{0,\,\pm(r+1)\bigr\}$ since
$\omega_{i}$ is assumed to be minuscule. Hence,
\begin{align*}
(\pi'(a_{q}'))(h_{i}) & =
\sum_{p=1}^{q} (\underbrace{a_{p}'-a_{p-1}'}_{> 0})
\nu_{p}'(h_{i}) \ge
-\sum_{p=1}^{q} (a_{p}'-a_{p-1}')(r+1) \\[2mm]
& = -(r+1)a_{q}'=-ra_{q} \ge -r.
\end{align*}
Thus, for every $0 \le q \le \ell$,
\begin{equation*}
\left((s-1)\omega_{i}+\pi'(a_{q}')\right)(h_{i})
\ge (s-1)-r = s-(r+1) \ge 0 \quad \text{by assumption},
\end{equation*}
which implies \eqref{eq:i}.
Thus we have proved the proposition. \qed

%%%%%%%%%%%%%%%%%%%%%%%%%%%%%%%%%%%%%%%%%%%%%%%%%%%%%%%%%%%%%%%%%%%%%%%%%%%%%%%%%%%%%%%%%%%%%%%%%%%%%%%%%%%%%%%%%%%%%%%%%%%%%%
%%%%%%%%%%%%%%%%%%%%%%%%%%%%%Section 4%%%%%%%%%%%%%%%%%%%%%%%%%%%%%%%%%%%%%%%%%%%%%%%%%%%%%%%%%%%%%%%%%%%%%%%%%%

\section{ The poset \texorpdfstring{$P^+(\lambda,2)/\sim$}{}}\label{section5}
 As we remarked earlier, it is clear that if $\sigma\in S_k$ then $\blambda$ and $\sigma\blambda$ are in the same equivalence class with respect to $\sim$ for all $\blambda\in P^+(\lambda,k)$. However the following example shows that outside $\lie{sl}_2$  the equivalence class of $\sim$ is in general bigger than the $S_k$ orbit of an element. Suppose that $\lie g$ is of type $\lie{sl}_3$, $k=3$ and $\lambda=3\omega_1+3\omega_2$. Then it is easily seen that
 $$\blambda=( \omega_2, \omega_1 + 2\omega_2, 2 \omega_1)\sim \bmu= ( 2 \omega_1,\omega_1, 2 \omega_1 + \omega_2),$$
 but  clearly $\blambda$ and $\bmu$ are not in the same $S_3$ orbit. However, when $k=2$, we prove below, Lemma~\ref{equiv2} that for all simple Lie algebras, the equivalence class is exactly the $S_2$ orbit.

\subsection{}\label{equiv-des} We begin with an equivalent formulation of the preorder in the case $k=2$.
\begin{prop}\label{prop4-1} Let $\lie g$ be a finite-dimensional simple complex Lie algebra and let $\blambda=(\lambda_1,\lambda_2)$ and $\bmu=(\mu_1,\mu_2)$ be elements of  $P^+(\lambda,2)/\sim$ for some $\lambda\in P^+$.
Then
\begin{gather*}\blambda\preceq\bmu\iff (\lambda_1-\mu_1)(h_\alpha)(\mu_1-\lambda_2)(h_\alpha) \geq 0 ,\  \ {\rm{for\ all}}\ \ \alpha\in R^+,\\ \iff (\lambda_1-\mu_1)(h_\alpha)(\mu_1-\lambda_2)(h_\alpha) \geq 0,\  \ {\rm{for\ all}}\ \ \alpha\in R. \end{gather*}
In particular, if $\bow\in W$ is such that $\bow(\lambda_1-\lambda_2)\in P^+$, then
$$\blambda\preceq\bmu\iff  \bow(\lambda_1-\mu_1)\in P^+\ \ {\rm{and}}\ \ \bow(\mu_1-\lambda_2)\in P^+.$$
\end{prop}
\begin{pf}
Since $r_{\alpha, 2}(\bmu)=\lambda(h_\alpha) $, we see that
\begin{equation*} r_{\alpha,2}(\bmu)-2r_{\alpha,1}(\bmu)=\begin{cases} (\mu_1-\mu_2)(h_\alpha),\ \ {\rm{if}}\ \ \mu_2(h_\alpha)\le \mu_1(h_\alpha),\\
(\mu_2-\mu_1)(h_\alpha),\ \ {\rm{otherwise}},\end{cases}\end{equation*}
or in other words that
\begin{equation}\label{alt} r_{\alpha,2}(\bmu)-2r_{\alpha,1}(\bmu)=|(\mu_1-\mu_2)(h_\alpha)|.\end{equation}
 Since $k=2$, we see that
 $$\blambda\preceq\bmu\iff r_{\alpha,1}(\blambda)\le r_{\alpha,1}(\bmu) \text{ for all } \alpha \in R^+,$$
 and hence we get
\begin{align*}
\blambda\preceq\bmu
 & \iff
  r_{\alpha,2}(\bmu)-2r_{\alpha,1}(\bmu) \le
  r_{\alpha,2}(\blambda)-2r_{\alpha,1}(\blambda) \\
 & \iff |(\mu_1-\mu_2)(h_\alpha)|\le |(\lambda_1-\lambda_2)(h_\alpha)| \\
 & \iff (\mu_1-\mu_2)(h_\alpha)^2\le (\lambda_1-\lambda_2)(h_\alpha)^2 \\
 & \iff (2\mu_1-\lambda_1-\lambda_2)(h_\alpha)^2\le (\lambda_1-\lambda_2)(h_\alpha)^2 \\
 & \iff (\lambda_1-\mu_1)(h_\alpha)(\mu_1-\lambda_2)(h_\alpha) \geq 0.
\end{align*}
for all $\alpha\in R^+$. Since $h_{-\alpha}=-h_\alpha$, we see that
$$\blambda\preceq\bmu \iff (\lambda_1-\mu_1)(h_\alpha)(\mu_1-\lambda_2)(h_\alpha) \ge 0
 \quad \text{\rm for all $\alpha \in R$}.$$

Now, let  $\bow\in W$ be such that $\bow(\lambda_1-\lambda_2)\in P^+$.
If $\blambda\preceq\bmu$, then for all $\alpha \in R^{+}$,
$$\bow(\lambda_1-\mu_1)(h_\alpha)\bow(\mu_1-\lambda_2)(h_\alpha)=
(\lambda_1-\mu_1)(h_{\bow^{-1}\alpha})(\mu_1-\lambda_2)(h_{\bow^{-1}\alpha}) \ge 0$$
by the first statement of the proposition. Also, we have
$$\bow(\lambda_1-\mu_1)(h_\alpha)+\bow(\mu_1-\lambda_2)(h_\alpha)
=\bow(\lambda_{1}-\lambda_{2})(h_{\alpha}) \ge 0 \quad
\text{for all $\alpha \in R^{+}$}$$
since $\bow(\lambda_1-\lambda_2)\in P^+$ by assumption.
Thus we conclude that $\bow(\lambda_1-\mu_1)(h_\alpha) \ge 0$ and
$\bow(\mu_1-\lambda_2)(h_{\alpha}) \ge 0$
for all $\alpha \in R^{+}$, which implies that
both of $\bow(\lambda_1-\mu_1)$ and
$\bow(\mu_1-\lambda_2)$ are dominant.
Conversely, assume that both of $\bow(\lambda_1-\mu_1)$ and
$\bow(\mu_1-\lambda_2)$ are dominant. Then, for all $\alpha \in R$,
$$(\lambda_1-\mu_1)(h_{\alpha})(\mu_1-\lambda_2)(h_{\alpha})=
\bow(\lambda_1-\mu_1)(h_{\bow^{-1}\alpha})
\bow(\mu_1-\lambda_2)(h_{\bow^{-1}\alpha}) \ge 0.$$
Hence, by the first statement of the proposition, we have
$\blambda\preceq\bmu$. Thus the second statement of
the proposition is established.
\end{pf}

%%%%%%%%%%%%%
\subsection{}
%%%%%%%%%%%%%
%
The next result gives information about the maximal elements in $P^+(\lambda, 2)/\sim$.
\begin{lem} \label{max}
Let $\lambda\in P^+$ and let $i\in I$, $\bow\in W$ be such that
$\lambda-\bow^{-1}\omega_i\in  P^+$.
Then the equivalence classes of $(\lambda,\lambda)$ and
$(\lambda,\lambda-\bow^{-1}\omega_i)$  are maximal
in the poset $P^+(2\lambda,2)$ and $P^+(2\lambda-\bow^{-1}\omega_i,2)$,
respectively.
\end{lem}
\begin{pf}
Suppose that $\bmu=(\mu_1,\mu_2)\in P^+(2\lambda,2)$ is
such that $(\lambda,\lambda)\preceq\bmu$ in $P^+(2\lambda,2)/\sim$.
Using Proposition \ref{equiv-des}, we get $\lambda-\mu_1\in P^+$ and
$\mu_1-\lambda\in P^+$ which forces $\mu_1=\mu_2=\lambda$ as required.

Similarly, if $\bmu\in P^+(2\lambda-\bow^{-1}\omega_i,2)/\sim$ with
$(\lambda,\lambda-\bow^{-1}\omega_i)\preceq\bmu$,
then Proposition \ref{equiv-des} gives, $\bow(\lambda-\mu_1)\in P^+$ and
$\bow(\mu_1-\lambda)+\omega_i\in P^+$. But this is only possible
if either $\mu_1=\lambda_1$ or $\mu_1-\lambda_1=-\bow^{-1}\omega_i$.
In either case, this implies that $\bmu=(\lambda,\lambda-\bow^{-1}\omega_i)$ in
$P^+(2\lambda-\bow^{-1}\omega_i, 2)/\sim$.
\end{pf}
%

%%%%%%%%%%%%%
\subsection{}
%%%%%%%%%%%%%
%
{{%
Suppose that $\lie g$ is of type $A_n$.
Then we can refine the preceding result as follows.
Given $\lambda=\sum_{i=1}^nr_i\omega_i\in P^+$,
define elements $\lambda^s$, $s=1,2$ as follows.
If $r_i\in2\bz_+$ for all $i\in I$, then take $\lambda^1=\lambda^2=\lambda$.
Otherwise let $1\le i_0< i_1\cdots <i_p\le n$ be
the set where $r_i$ is odd and set $I_+=I\setminus\{i_0,\cdots,i_p\}$.
Define
$$\lambda^1=
  \sum_{s=0}^p((r_{i_s}+ (-1)^s)/2)\omega_{i_s}+\sum_{i\in I_+}(r_i/2)\omega_i, \qquad
  \lambda^2=\lambda-\lambda^1.$$

In either case, set $\blambda_{\max}=(\lambda^1,\lambda^2)$.

\begin{prop}\label{maxslr} Let $\lambda\in P^+$ and $\lie g$ be of type $A_n$. Then either $\lambda^1=\lambda^2$ or $\lambda^2=\lambda^1-\bow^{-1}\omega_i$ for some $\bow\in W$ and $i\in I$. In either case, $\blambda_{\max}$ is the unique maximal element in $P^+(\lambda,2)/\sim$.
\end{prop}
\begin{pf} If $\lambda^1\ne\lambda^2$, then by definition, we have $$\lambda^1-\lambda^2= \omega_{i_0}-\omega_{i_1}+\cdots +(-1)^p\omega_{i_p},$$ where $0\le i_0<\cdots< i_p\le n$. It is elementary to see that $(\lambda^1-\lambda^2)(h_\alpha)\in\{0,\pm 1\}$, i.e., $\lambda^1-\lambda^2$ is in $W \tau$ for some minuscule $\tau \in P^+$, hence in $W \omega_i$ for some $i \in I$.
It remains to prove that it is the unique maximal element.
In other words, we must prove that if $\bmu\in P^+(\lambda,2)$
then $\bmu\preceq\blambda_{\max}$. Again, using Proposition \ref{equiv-des}
it suffices to prove that
$$(\mu_1-\lambda^1)(h_\alpha)(\mu_1-\lambda^2)(h_\alpha)=(\mu_1-\lambda^1)(h_\alpha)(\lambda^1-\mu_2)(h_\alpha)\ge 0.$$
If $(\mu_1-\lambda^1)(h_\alpha)=0$  there is nothing to prove.
 If $(\mu_1-\lambda^1)(h_\alpha)>0$ then
 since $(\lambda^1-\lambda^2)(h_\alpha)\in\{0,\pm 1\}$ we get
   $\mu_1(h_\alpha)\ge\lambda^2(h_\alpha)$ as required. The case when $(\mu_1-\lambda^1)(h_\alpha)<0$ is identical.
\end{pf}}}

%%%%%%%%%%%%%
\subsection{}
%%%%%%%%%%%%%
%
%%%%%%%%%%%%%%%%
%%% prop:mid %%%
%%%%%%%%%%%%%%%%
%
\begin{prop} \label{prop:mid}
Let $\blambda, \bmu \in P^+(\lambda,2)/\sim$ with $\blambda\prec\bmu$ and
assume there exists  $\bow\in W$ and $i_0\in I$ such that $\bow(\lambda_1-\lambda_2)\in P^+$ and
\begin{equation} \label{support}
\bow(\lambda_1-\mu_1)(h_{i_0})\bow(\mu_1-\lambda_2)(h_{i_0}) >0.
\end{equation}
Then, $(\lambda_1-\bow^{-1}\omega_{i_0},\lambda_2+\bow^{-1}\omega_{i_0})\in P^+(\lambda,2)$ and
\begin{equation} \label{cover2}
\blambda \prec (\lambda_1-\bow^{-1}\omega_{i_0},\lambda_2+\bow^{-1}\omega_{i_0}) \preceq \bmu.
\end{equation}
\end{prop}

% First, that \eqref{support} gives $\bow(\lambda_1 - \lambda_2)(h_{i_0}) > 0$.

\begin{pf}
First we remark that by \eqref{support} and the assumption that
$\bow(\lambda_1-\lambda_2) \in P^+$,
\begin{equation} \label{mid01}
\bow(\lambda_1-\lambda_2)(h_{i_{0}}) =
 \bow(\lambda_1-\mu_1)(h_{i_{0}})+\bow(\mu_1-\lambda_2)(h_{i_{0}}) > 0.
\end{equation}

Let us show that $\lambda_1-\bow^{-1}\omega_{i_0}$ and
$\lambda_2+\bow^{-1}\omega_{i_0}$ are dominant,
which implies that
$(\lambda_1-\bow^{-1}\omega_{i_0},\lambda_2+\bow^{-1}\omega_{i_0})\in P^+(\lambda,2)$.
For $j\in I$, write $$\bow h_j=\sum_{i=1}^n r_ih_i,$$ and
note that either $r_i\ge 0$ for all $i \in I$ or $r_i\le 0$ for all $i\in I$.
If $r_{i_0}\le  0$, then obviously $(\lambda_1- \bow^{-1}\omega_{i_0})(h_j)\ge 0$.
If $r_i \ge 0$ for all $i\in I$, then we have
$$(\lambda_1-\lambda_2)(h_j)=
 \bow(\lambda_1-\lambda_2)(\bow h_j)
 \ge r_{i_0}\bow(\lambda_1-\lambda_2)(h_{i_0}) \ge r_{i_0},$$
where the first inequality follows from the assumption
that $\bow(\lambda_1-\lambda_2)\in P^+$, and the second inequality
follows from \eqref{mid01}.
Hence, $\lambda_1(h_j) \ge \lambda_{2}(h_{j})+r_{i_0} \ge r_{i_{0}}$, and
hence $(\lambda_1-\bow^{-1}\omega_{i_0})(h_{j}) \ge 0$
since $\bow^{-1}\omega_{i_0}(h_j)=r_{i_0}$.
Thus we have proved that $\lambda_1-\bow^{-1}\omega_{i_0}\in P^+$.
To prove that $\lambda_2+\bow^{-1}\omega_{i_0}\in P^+$,
we note that if $r_{i_0}\ge 0$ there is nothing to prove.
If $r_i \le 0$ for all $i\in I$, then we have
$$(\lambda_1-\lambda_2)(h_j)=\bow(\lambda_1-\lambda_2)(\bow h_j) \le
 r_{i_0}\bow(\lambda_1-\lambda_2)(h_{i_0})\le r_{i_0},$$
where the first inequality follows from the assumption
that $\bow(\lambda_1-\lambda_2)\in P^+$, and the second inequality
follows from \eqref{mid01}.
Hence, $\lambda_2(h_j) \ge \lambda_{1}(h_{j})-r_{i_0} \ge -r_{i_{0}}$,
and so
$(\lambda_2+\bow^{-1}\omega_{i_0})(h_j)=\lambda_2(h_j)+r_{i_0}\ge 0$,
proving that $\lambda_2+\bow^{-1}\omega_i\in P^+$.

By Proposition \ref{equiv-des},
we see that  $\bow(\lambda_1-\mu_1)$ and $ \bow(\mu_1-\lambda_2)$ are in $P^+$.
Hence \eqref{support} gives that
$$\bow(\lambda_1-\mu_1)(h_{i_0}) > 0, \quad
\bow(\mu_1-\lambda_2)(h_{i_0}) > 0, \quad \text{and hence} \quad
\bow(\lambda_1-\lambda_2)(h_{i_0}) > 0,$$
which in turn gives that
\begin{equation} \label{cover1}
\bow(\lambda_1-\mu_1)-\omega_{i_0}, \
\bow(\mu_1-\lambda_2)-\omega_{i_0}, \
\bow(\lambda_1-\lambda_2)-\omega_{i_0} \in P^+.
\end{equation}

To prove
$$\blambda \preceq (\lambda_1-\bow^{-1}\omega_{i_0},\lambda_2+\bow^{-1}\omega_{i_0}) \preceq \bmu,$$
we must show that for all $\alpha\in R^+$,
\begin{align*}
& \bow^{-1}\omega_{i_0}(h_\alpha)
  (\lambda_1-\bow^{-1}\omega_{i_0} -\lambda_2)(h_\alpha) \ge 0, \\
& (\lambda_1-\bow^{-1}\omega_{i_0} -\mu_1)(h_\alpha)
  (\mu_1-\lambda_2-\bow^{-1}\omega_{i_0})(h_\alpha)\ge 0,
\end{align*}
or equivalently that
\begin{align*}
& \omega_{i_0}(h_\alpha)(\bow(\lambda_1-\lambda_2)-\omega_{i_0})(h_\alpha)\ge 0, \\
& (\bow(\lambda_1-\mu_1)-\omega_{i_0})(h_\alpha)
  (\bow(\mu_1-\lambda_2)-\omega_{i_0})(h_\alpha) \ge 0.
\end{align*}
 But this is now immediate from \eqref{cover1}.

In order to prove \eqref{cover2}, it remains to show $\blambda \prec (\lambda_1-\bow^{-1}\omega_{i_0},\lambda_2+\bow^{-1}\omega_{i_0})$. For that notice that $\lambda_1-\bow^{-1}\omega_{i_0}\notin\{\lambda_1,\lambda_2\}$ since
then by Lemma \ref{max},  we would have that $\blambda$ is
a maximal element of $P^+(\lambda,2)$ and this would contradict the fact
that $\blambda\prec\bmu$.
\end{pf}

%%%%%%%%%%%%%
\subsection{}
%%%%%%%%%%%%%

In this section, we will show that for $k=2$,
equivalence classes in $P^+(\lambda,2)$ are
the $S_2$-orbits, generalizing the results of Section~\ref{section-classes}.

\begin{lem}\label{equiv2}
Let $\lie g$ be arbitrary and $\blambda,\bmu\in P^+(\lambda,2)$
for some $\lambda\in P^+$. Then $\blambda\sim\bmu$ iff
$\bmu$ and $\blambda$ are in the same $S_2$-orbit.
\end{lem}

\begin{pf} Suppose that $\blambda=(\lambda_1,\lambda_2)$ and
$\bmu=(\mu_1,\mu_2)$ and set $\nu = \lambda_1 - \lambda_2$ and $\nu' = \mu_1 - \mu_2$.
If $\blambda\sim\bmu$ then we see from \eqref{alt}
that for all $\alpha\in R$,
we have $\nu(h_\alpha)=\pm \nu'(h_\alpha)$, where the sign depends on $\alpha$.
It suffices to show that we can choose the sign consistently.
Suppose for a contradiction that this is not so,
then there exists a  connected subset $I_0$ of $I$ and $i_1,i_2\in I_0$
such that
$$\nu(h_{i_1})=\nu'(h_{i_1}) \neq 0, \quad
  \nu(h_{i_2})=-\nu'(h_{i_2}) \neq 0, \quad
  \nu(h_j)=\nu'(h_j)=0,\ \ j\in I_0 \setminus \{ i_1,i_2\}.$$
Set $\beta=\sum_{i\in I_0}\alpha_i$; we can easily check that
$\beta$ is a (positive) root, i.e., $\beta \in R^{+}$. Then,
$$\nu(h_\beta)=\nu(h_{i_1})+\nu(h_{i_2})=
  \nu'(h_{i_1})-\nu'(h_{i_2}) \ne
  \pm\left(\nu'(h_{i_1})+\nu'(h_{i_2})\right).$$
Since $\beta\in R^+$, we get the required contradiction.
\end{pf}

%%%%%%%%%%%%%%%%%%%%%%%%%%%%%%%%%%%%%%%%%%%%%%%%%%%%%%%%%%%%%%%%%%%%%%%%%
%%%%%%%%%%%%%%%%%%%%%%%%%%%%%%%%Section 5%%%%%%%%%%%%%%%%%%%%%%%%%%%%%%%%

%%%%%%%%
\section{Proof of \texorpdfstring{Theorem~\ref{mainthm}\,{\rm (iii)}}{Theorem 1 (iii)}}
%%%%%%%%
%
\label{section6}

In this section, we assume that $\lie g$ is of type $A_2$ and
prove Theorem~\ref{mainthm}~(iii). We begin by showing that we can restrict our attention to certain elements $\blambda$ and $\bmu$ of $P^+(\lambda,2)$.

%%%%%%%%%%%%%
%About the sufficient cover elements
%%%%%%%%%%%%%
%

\subsection{}  Since the poset $P^+(\lambda,2)$ is finite,  that it suffices to prove part (iii) of Theorem \ref{mainthm} for $\blambda$ and $\bmu$ where $\bmu$ is a cover of $\blambda$; i.e $\blambda\prec\bmu$ and there does not exist $\bnu\in P^+(\lambda,2)$ with $\blambda\prec\bnu\prec\bmu$. We  first show  that in fact it suffices to prove Theorem \ref{mainthm} (iii) for certain special $\blambda$  and also that for these $\blambda$ we can restrict our attention to certain special covers.

We shall use freely the following two facts. The first is well-known.
$$V(\lambda)\otimes V(\mu)\cong_{\lie g}V(\mu)\otimes V(\lambda).$$
The second fact is that the partial order on $P^+(\lambda, k)$ is
compatible with duals (see \eqref{dual} and \eqref{w0}) and
that for all $\lambda,\mu,\nu\in P^+$, we have
$$\dim \Hom_{\lie g}(V(\nu),\,V(\lambda)\otimes V(\mu)) =
  \dim \Hom_{\lie g}(V(-\bow_0\nu),\,V(-\bow_0\lambda) \otimes V(-\bow_0\mu)).$$
This allows us to switch freely between proving Theorem \ref{mainthm}(iii)
either for $\blambda \prec \bmu$ or for $-\bow_0\blambda\prec -\bow_0\bmu$.
Recall that $-\bow_0\omega_1=\omega_2$.

%%%%%%%%%%%%%

%%%%%%%%%%%%%%%%%%%%%%
%%% cover-elements %%%
%%%%%%%%%%%%%%%%%%%%%%
%
\begin{prop} \label{cover-elements}
Let $\lambda\in P^+$,
$\blambda=(\lambda_{1},\,\lambda_{2}) \in P^+(\lambda,2)$ and
assume that $\bmu=(\mu_1,\mu_2)$ covers $\blambda$.
It suffices to prove Theorem~1\,(iii) is true
when $\blambda$ and $\bmu$ satisfy the following conditions
\eqref{limitlambda} and \eqref{limitbmu}
for some $\bow \in \bigl\{\id,\, \bos_{1},\, \bos_{2}\bigr\}$:
\begin{gather} \label{limitlambda}
\bow(\lambda_{1}-\lambda_{2}) \in P^{+}, \quad
\bow(\lambda_1-\lambda_2)(h_1) > 0,
\end{gather}
and  either
\begin{gather} \label{limitbmu}
\begin{cases}
\bmu = (\lambda_{1}-\bow\omega_{1},\ \lambda_{2}+\bow\omega_{1}) \quad \text{\rm or} \\[1.5mm]
\bmu=(\lambda_1 - \bow(\lambda_1-\lambda_2)(h_1)\bow \omega_1, \
 \lambda_2 + \bow(\lambda_1-\lambda_2)(h_1)\bow \omega_1).
\end{cases}
\end{gather}
\end{prop}

\begin{proof} We first prove that we  can assume that $\blambda$ satisfies the conditions in \eqref{limitlambda}. Suppose that $\blambda=(\lambda_1,\lambda_2)\in P^+(\lambda,2)$ is such that $\lambda_1-\lambda_2\in P^+$ but $\lambda_1(h_1)=\lambda_2(h_1)$. Since $\blambda$ is not the maximal element in $P^+(\lambda,2)$ it follows from Lemma \ref{max} that $\lambda_1\ne \lambda_2$ and hence we must have $\lambda_1(h_2)>\lambda_2(h_2)$. We have $-\bow_0\blambda\prec -\bow_0\bmu$ and hence $-\bow_0(\lambda_1-\lambda_2)(h_1)>0$. If $\bos_2(\lambda_1-\lambda_2)\in P^+$ or $\bos_1(\lambda_1-\lambda_2)\in P^+$, a similar argument shows that either $\blambda$ or $-\bow_0\blambda$ satisfies the conditions in \eqref{limitlambda}.  Suppose now that $\bow(\lambda_1-\lambda_2)\in P^+$ but $\bow\notin\{\id,\bos_1,\bos_2\}$. Then $\bow\bow_0\in\{\id,\bos_1,\bos_2\}$ and hence  we  can work with the pair $(-\bow_0\lambda_2, -\bow_0\lambda_1)$.

We now prove that we can also assume that $\bmu$
satisfies the conditions in \eqref{limitbmu}.

%%%%%%%%%%%%%%%%%%%
\paragraph{Case 1.}
%%%%%%%%%%%%%%%%%%%
%
Suppose that there exists $i \in I=\bigl\{1,\,2\bigr\}$ such that
\begin{equation*}
\bow(\lambda_{1}-\mu_{1})(h_{i}) \bow(\mu_{1}-\lambda_{2})(h_{i}) > 0,
\end{equation*}
where $\bow \in \bigl\{\id,\, \bos_{1},\, \bos_{2}\bigr\}$.
We see from Proposition~\ref{prop4-1} that
$\bow(\lambda_{1}-\mu_{1})$ and
$\bow(\mu_{1}-\lambda_{2}) \in P^{+}$. Thus we have
$\bow(\lambda_{1}-\mu_{1})(h_{i}) > 0$ and
$\bow(\mu_{1}-\lambda_{2})(h_{i}) > 0$. In particular,
%
%%%%%%%%%%%%%%
%%% eq:pos %%%
%%%%%%%%%%%%%%
%
\begin{equation} \label{eq:pos}
\bow(\lambda_{1}-\lambda_{2})(h_{i}) =
\bow(\lambda_{1}-\mu_{1})(h_{i}) +
\bow(\mu_{1}-\lambda_{2})(h_{i}) > 0.
\end{equation}

%%%%%%%%%%%%%%%%%%%%%%%%
\paragraph{Subcase 1.1.}
%%%%%%%%%%%%%%%%%%%%%%%%
%
If $i=1$, then it follows from Proposition 5.4 that
\begin{equation*}
\blambda \prec
 (\lambda_1-\bow^{-1}\omega_1,\,\lambda_2+\bow^{-1}\omega_1)=
 (\lambda_1-\bow\omega_1,\,\lambda_2+\bow\omega_1)
 \preceq\bmu.
\end{equation*}
Since $\bmu$ covers $\blambda$, it follows that $\bmu=(\lambda_1-\bow\omega_1,\,\lambda_2+\bow\omega_1)$
as required.

%%%%%%%%%%%%%%%%%%%%%%%%
\paragraph{Subcase 1.2.}
%%%%%%%%%%%%%%%%%%%%%%%%
%
If $i=2$, then it follows from Proposition 5.4 that
\begin{equation*}
\blambda \prec
 (\lambda_1-\bow^{-1}\omega_2,\,\lambda_2+\bow^{-1}\omega_2)
 = (\lambda_1-\bow\omega_2,\,\lambda_2+\bow\omega_2)
 \preceq\bmu.
\end{equation*}
By the ``duality'', we get
\begin{equation*}
-\bow_{0}\blambda \prec
(-\bow_{0}\lambda_1-(-\bow_{0}\bow\omega_2),\,
 -\bow_{0}\lambda_2+(-\bow_{0}\bow\omega_2))
 \preceq -\bow_{0}\bmu.
\end{equation*}

Since $-\bow_{0}\bmu$ covers $-\bow_{0}\blambda$, we get
\begin{equation*}
-\bow_{0}\bmu =
(-\bow_{0}\lambda_1-(-\bow_{0}\bow\omega_2),\,
 -\bow_{0}\lambda_2+(-\bow_{0}\bow\omega_2)).
\end{equation*}
We set $\ti{\bow} = \bow_0 \bow \bow_0$ and note that we have
\begin{equation*}
\ti{\bow}:=
\begin{cases}
\id & \text{if $\bow=\id$}, \\[1mm]
\bos_{2} & \text{if $\bow=\bos_{1}$}, \\[1mm]
\bos_{1} & \text{if $\bow=\bos_{2}$}.
\end{cases}
\end{equation*}
Moreover we also have
\begin{equation*}
\ti{\bow}((-\bow_0\lambda_1)-(-\bow_{0}\lambda_2))=
-\bow_{0}\bow(\lambda_{1}-\lambda_{2}) \in P^{+},
\end{equation*}
\begin{equation*}
\ti{\bow}((-\bow_0\lambda_1)-(-\bow_{0}\lambda_2))(h_{1})=
-\bow_{0}\bow(\lambda_{1}-\lambda_{2})(h_{1}) =
\bow(\lambda_{1}-\lambda_{2})(h_{2}) > 0 \quad \text{by \eqref{eq:pos}}.
\end{equation*}
Hence, $-\bow_0\blambda$ and $-\bow_0\bmu$ satisfy the conditions
(with $\bow$ replaced by $\ti{\bow}$).
Hence, if Theorem~1\,(iii) is established for this pair,
then it follows for the pair $\blambda$ and $\bmu$ as discussed earlier.

%%%%%%%%%%%%%%%%%%%
\paragraph{Case 2.}
%%%%%%%%%%%%%%%%%%%
%
Suppose that
\begin{equation*}
\bow(\lambda_{1}-\mu_{1})(h_{i}) \bow(\mu_{1}-\lambda_{2})(h_{i}) \le 0
\quad \text{for all $i \in I=\bigl\{1,\,2\bigr\}$},
\end{equation*}
where $\bow \in \bigl\{\id,\, \bos_{1},\, \bos_{2}\bigr\}$.
We see from Proposition 5.1 that
$\bow(\lambda_{1}-\mu_{1}) \in P^{+}$ and
$\bow(\mu_{1}-\lambda_{2}) \in P^{+}$. Thus,
\begin{equation*}
\bow(\lambda_{1}-\mu_{1})(h_{i}) \bow(\mu_{1}-\lambda_{2})(h_{i}) = 0
\quad \text{for all $i \in I=\bigl\{1,\,2\bigr\}$},
\end{equation*}
which implies that $\bow\mu_{1}(h_{i})=\bow\lambda_{1}(h_{i})$
or $\bow\mu_{1}(h_{i})=\bow\lambda_{2}(h_{i})$ for each $i=1,\,2$.
Remark that $\blambda$ is not the maximal element, since
$\blambda \prec \bmu$. Therefore it follows that the only
possibilities are
\begin{equation*}
\bow\mu_1=(\bow\lambda_2)(h_1)\omega_1+ (\bow\lambda_1)(h_2)\omega_2
\quad \text{\rm or} \quad
\bow\mu_1=(\bow\lambda_1)(h_1)\omega_1+ (\bow\lambda_2)(h_2)\omega_2.
\end{equation*}
In turn this implies that
\begin{align*}
\bmu & \sim
\bigl( (\bow\lambda_2)(h_1)\bow^{-1}\omega_1+(\bow\lambda_1)(h_2)\bow^{-1}\omega_2, \
(\bow\lambda_1)(h_1)\bow^{-1}\omega_1 +(\bow\lambda_2)(h_2)\bow^{-1}\omega_2 \bigr) \\
& = (\lambda_1 - \bow(\lambda_1-\lambda_2)(h_1)\bow^{-1} \omega_1, \
 \lambda_2 + \bow(\lambda_1-\lambda_2)(h_1)\bow^{-1} \omega_1) \\
& = (\lambda_1 - \bow(\lambda_1-\lambda_2)(h_1)\bow \omega_1, \
 \lambda_2 + \bow(\lambda_1-\lambda_2)(h_1)\bow \omega_1),
\end{align*}
as required; here we use the fact that
\begin{equation*}
\begin{cases}
\lambda_{1}=
(\bow\lambda_1)(h_1)\bow^{-1}\omega_1+(\bow\lambda_1)(h_2)\bow^{-1}\omega_2, \\[1mm]
\lambda_{2}=
(\bow\lambda_2)(h_1)\bow^{-1}\omega_1+(\bow\lambda_2)(h_2)\bow^{-1}\omega_2,
\end{cases}
\end{equation*}
and then the fact that $\bow^{-1}=\bow$.
\end{proof}

%%%%%%%%%%%%%
\subsection{} \label{nak}
%%%%%%%%%%%%%
%
We now recall from  \cite{KN1994}, \cite{N2002}
a tableaux  description of  tensor product multiplicities.
Given $\lambda \in P^+$, let $\BT(\lambda)\subset \bz_+^5$ be
the subset consisting of tuples
$(s_{1,1}, s_{1,2}, s_{1,3}, s_{2,2}, s_{2,3})$ satisfying the conditions,
\begin{equation} \label{sst1}
 s_{1,1} + s_{1,2} + s_{1,3} = \lambda(h_1) + \lambda(h_2), \qquad
 s_{2,2} + s_{2,3} = \lambda(h_2),
\end{equation}
\begin{equation} \label{sst2}
 s_{1,1} \geq s_{2,2}, \qquad
 s_{1,1} + s_{1,2} \geq s_{2,2} + s_{2,3}.
\end{equation}
Then, it is proved in \cite{KN1994} that
$$\dim V(\lambda)=\#\BT(\lambda).$$
(This is just the number of semistandard tableaux
with entries from $\{1,2,3\}$ of shape $\lambda$,
{where $s_{i,j}$ corresponds to
the number of $j$ in the $i$-th row}).
Moreover, if $\nu \in P$ and we set
$$\BT(\lambda)^\nu=\{(s_{i,j})\in \BT(\lambda):
{ s_{1,1} - s_{1,2} - s_{2,2} = \nu(h_1), \,
s_{1,2} + s_{2,2} - s_{1,3} - s_{2,3} = \nu(h_2)}\},$$
then
$$\dim V(\lambda)_\nu =\# \BT(\lambda)^\nu.$$
In particular, { if $(s_{i,j}), (t_{i,j}) \in \BT(\lambda)^\nu$,}
then they satisfy
\begin{equation} \label{sameweight}
s_{1,1}=t_{1,1}, \ \ s_{1,2}+s_{2,2}=t_{1,2}+t_{2,2}, \ \ \
s_{1,3}+s_{2,3}= t_{1,3}+t_{2,3}.
\end{equation}

Suppose now that $\mu, \nu \in P^+$, then (\cite{N2002})
\begin{equation}
\dim\Hom_{\lie g}(V(\nu), V(\mu)\otimes V(\lambda)) = \# \BT(\lambda)_\mu^\nu,
\end{equation}
where $\BT(\lambda)_\mu^\nu$ is the subset of
$\BT(\lambda)^\nu$, consisting of
$(s_{1,1}, s_{1,2}, s_{1,3}, s_{2,2}, s_{2,3})\in \BT(\lambda)$
satisfying the following additional constraints:
\begin{equation} \label{LR1}
 s_{1,2}\le \mu(h_1), \qquad
 s_{1,3}\le \mu(h_2), \qquad
 s_{2,3}+s_{1,3}\le \mu(h_2)+s_{1,2},
\end{equation}
\begin{equation} \label{LR2}
\nu(h_1)+\nu(h_2)=\mu(h_1)+\mu(h_2)+s_{1,1}-s_{1,3}-s_{2,3},
\end{equation}
\begin{equation} \label{LR3}
\nu(h_2)=\mu(h_2)+s_{1,2} + s_{2,2} - s_{1,3} - s_{2,3}.
\end{equation}
As a consequence, we see that to prove Theorem \ref{mainthm}(iii),
we must prove that if $\blambda,\bmu\in P^+(\lambda,2)$, then
\begin{equation} \label{equivthmk2}
\blambda \preceq \bmu \implies
\# \BT(\lambda_2)_{\lambda_1}^\nu \le \# \BT(\mu_2)_{\mu_1}^\nu
\quad \text{for each $\nu \in P^+$}.
\end{equation}
This is done in the rest of the section.

%%%%%%%%%%%%%
\subsection{}
%%%%%%%%%%%%%

Keep the notation in Proposition~\ref{cover-elements}.
In this subsection, we prove that Theorem~\ref{mainthm}\,(iii) is true
if $\blambda$ and $\bmu$ satisfy the conditions \eqref{limitlambda} and
\eqref{limitbmu} with $\bow = \id$ or $\bow= \bos_2$. By \eqref{equivthmk2},
it suffices to find an injective map from
\begin{equation} \label{eq:t1k}
\BT(\lambda_2)_{\lambda_1}^\nu \hookrightarrow
\BT(\lambda_2+a\bow\omega_1)_{\lambda_1-a\bow\omega_1}^\nu=
\BT(\lambda_2+a\omega_1)_{\lambda_1-a\omega_1}^\nu
\end{equation}
for each $\nu \in P^{+}$, where $a$ equals either
$1$ or $\bow(\lambda_1 - \lambda_2)(h_1)$; note that
$\bow(\lambda_1 - \lambda_2)(h_1) > 0$ by
the second equality of \eqref{limitlambda}.
This is obtained as a corollary of the following proposition.

%%%%%%%%%%%
%%% t1k %%%
%%%%%%%%%%%
%
\begin{prop} \label{t1k}
Keep the notation above. For each $\nu \in P^{+}$,
there exists $0 \le \ell \le a$ such that for all
$(s_{i,j}) \in \BT(\lambda_2)_{\lambda_1}^\nu$,
we have
\begin{equation*}
a-\ell \le s_{1,2} \le \lambda_1(h_1) - \ell, \qquad
s_{1,3} \le \lambda_1(h_2) - (a - \ell), \qquad
s_{2,3} \ge a - \ell.
\end{equation*}
\end{prop}

\begin{proof}
First, let us show that
\begin{equation}\label{bounds}
\lambda_1(h_1)+\lambda_1(h_2)-a \ge \lambda_2(h_1)+\lambda_2(h_2), \quad
\lambda_1(h_1)-\lambda_2(h_1) \ge a.
\end{equation}
Indeed, since $\bow(\lambda_{1}-\lambda_{2})(h_{1}) \ge a$
by the definition of $a$, we have
\begin{equation*}
\begin{cases}
(\lambda_{1}-\lambda_{2})(h_{1}) \ge a & \text{if $\bow=\id$}, \\[1.5mm]
(\lambda_{1}-\lambda_{2})(h_{1}+h_{2}) \ge a & \text{if $\bow=\bos_{2}$},
\end{cases}
\end{equation*}
which implies the second (resp., first) inequality of \eqref{bounds}
if $\bow=\id$ (resp., $\bow=\bos_{2}$). Also, since
$\bow(\lambda_{1}-\lambda_{2}) \in P^{+}$, we see that
\begin{equation*}
\bow(\lambda_{1}-\lambda_{2})(h_{1}+h_{2}) \ge \bow(\lambda_{1}-\lambda_{2})(h_{1}) \ge a.
\end{equation*}
Thus we get
\begin{equation*}
\begin{cases}
(\lambda_{1}-\lambda_{2})(h_{1}+h_{2}) \ge a & \text{if $\bow=\id$}, \\[1.5mm]
(\lambda_{1}-\lambda_{2})(h_{1}) \ge a & \text{if $\bow=\bos_{2}$},
\end{cases}
\end{equation*}
which implies the first (resp., second) inequality of \eqref{bounds}
if $\bow=\id$ (resp., $\bow=\bos_{2}$).

By \eqref{LR1}, we have $t_{1,2} \leq \lambda_1(h_1)$ for all
$(t_{i,j}) \in \BT(\lambda_2)_{\lambda_1}^{\nu}$. Thus we can choose
$0 \leq \ell \leq a$ maximal such that
for all $(t_{i,j}) \in \BT(\lambda_2)_{\lambda_1}^{\nu}$,
$$t_{1,2} \leq \lambda_1(h_1) - \ell.$$
In particular, we can and do fix an element
$(s_{i,j}) \in \BT(\lambda_2)^\nu_{\lambda_1}$ with $s_{1,2}=\lambda_1(h_1)-\ell$.

Suppose that there exists $(t_{i,j}) \in \BT(\lambda_2)_{\lambda_1}^\nu$,
with $t_{1,3} > \lambda_1(h_2) + \ell - a$. Then, \eqref{sameweight} gives
$$t_{2,2} + t_{1,2} = s_{2,2} + s_{1,2} \geq s_{1,2} = \lambda_1(h_1) - \ell.$$
This implies that
\begin{align*}
\lambda_2(h_1) + \lambda_2(h_2)
 & = t_{1,1} + t_{1,2} + t_{1,3} \quad \text{by the first equality of \eqref{sst1}} \\
 & \geq t_{2,2} + t_{1,2} + t_{1,3} \quad \text{by the first inequality of \eqref{sst2}} \\
 & \geq \lambda_1(h_1) - \ell  + t_{1,3} \\
 & > \lambda_1(h_1) - \ell + \lambda_1(h_2) + \ell - a \\
 & = \lambda_1(h_1) + \lambda_1(h_2) - a.
\end{align*}
This contradicts the inequality
$\lambda_2(h_1) + \lambda_2(h_2) \le
\lambda_1(h_1) + \lambda_1(h_2) - a$ obtained in \eqref{bounds}.

Suppose now that there exists $(t_{i,j}) \in \BT(\lambda_2)_{\lambda_1}^\nu$
with $t_{1,2}< a-\ell$. Then we have
\begin{align*}
\lambda_2(h_2)
  & = t_{2,2} + t_{2,3} \quad \text{by the second equality of \eqref{sst1}} \\
  & \leq t_{1,1} + t_{1,2} \quad \text{by the second inequality of \eqref{sst2}} \\
  & < t_{1,1} + a - \ell = s_{1,1} + a - \ell
    \quad \text{by the first equality of \eqref{sameweight}} \\
  & \leq s_{1,1} + s_{1,3} + a - \ell.
\end{align*}
On the other hand, we have
\begin{align}
s_{1,1} + s_{1,3} + a - \ell
 & = \lambda_2(h_1) + \lambda_2(h_2) - s_{1,2} + a - \ell
   \quad \text{by the first equality of \eqref{sst1}} \nonumber \\
 & = \lambda_2(h_1) + \lambda_2(h_2) - (\lambda_1(h_1) - \ell) + a - \ell \nonumber \\
 & = \lambda_2(h_1) + \lambda_2(h_2) - \lambda_1(h_1) + a. \label{proof5-4}
\end{align}
Combining the two, gives
$$\lambda_2(h_2) < \lambda_2(h_1) + \lambda_2(h_2) - \lambda_1(h_1) + a
  \text{ and so } \lambda_1(h_1) - \lambda_2(h_1) < a,$$
which contradicts the second inequality of \eqref{bounds}.

Finally suppose that there exists
$(t_{i,j})\in \BT(\lambda_2)_{\lambda_1}^\nu$ with $t_{2,3} < a-\ell$.
Then we have
\begin{align*}
\lambda_2(h_2) & = t_{2,2} + t_{2,3}
  \leq t_{1,1} + t_{2,3} < t_{1,1} + a - \ell =
  s_{1,1} + a - \ell \leq s_{1,1} + s_{1,3} + a - \ell.
\end{align*}
Since $s_{1,1} + s_{1,3} + a - \ell = \lambda_2(h_1) + \lambda_2(h_2) - \lambda_1(h_1) + a$
by \eqref{proof5-4}, we get
$$\lambda_2(h_2)<\lambda_2(h_1)+\lambda_2(h_2)-\lambda_1(h_1)+a,$$
which again contradicts the second inequality of \eqref{bounds}.
\end{proof}

The following corollary is now trivially checked using Subsection~\ref{nak}.
Thus we have proved that Theorem~\ref{mainthm}(iii) is true
if $\blambda$ and $\bmu$ satisfy the conditions \eqref{limitlambda} and
\eqref{limitbmu} with $\bow = \id$ or $\bow= \bos_2$ (see \eqref{eq:t1k}).
%
%%%%%%%%%%%%%%%
%%% cor-t1k %%%
%%%%%%%%%%%%%%%
%
\begin{cor} \label{cor-t1k}
Keep the notation and setting in the proposition above.
Let $\nu \in P^{+}$, and let $\ell$ be as in the proposition above.
Then, the assignment $(s_{i,j}) \mapsto (s'_{i,j})$,
$$s'_{1,1} = s_{1,1} + a, \qquad s'_{1,2} = s_{1,2} - (a -  \ell), \qquad
  s'_{1,3} = s_{1,3} + (a - \ell)$$
$$s'_{2,2} = s_{2,2} + (a -\ell), \qquad
  s'_{2,3} = s_{2,3} - (a - \ell),$$
defines an injective map
$\BT(\lambda_2)^\nu_{\lambda_1} \hookrightarrow
\BT(\lambda_2+ a\omega_1)^\nu_{\lambda_1-a\omega_1}$.
\end{cor}

%%%%%%%%%%%%%
\subsection{}
%%%%%%%%%%%%%

Again, keep the notation in Proposition~\ref{cover-elements}.
In this subsection, we prove that Theorem~\ref{mainthm}\,(iii) is true
if $\blambda$ and $\bmu$ satisfy the conditions \eqref{limitlambda} and
\eqref{limitbmu} with $\bow = \bos_1$. By \eqref{equivthmk2},
it suffices to find an injective map from
\begin{equation} \label{eq:t2k}
\BT(\lambda_2)_{\lambda_1}^\nu \hookrightarrow
\BT(\lambda_2+a\bos_1\omega_1)_{\lambda_1-a\bos_1\omega_1}^\nu=
\BT(\lambda_2+a(\omega_2-\omega_1))_{\lambda_1-a(\omega_2-\omega_1)}^\nu
\end{equation}
for each $\nu \in P^{+}$, where $a$ equals either
$1$ or $\bos_1(\lambda_1 - \lambda_2)(h_1)$; note that
$\bos_1(\lambda_1 - \lambda_2)(h_1) > 0$ by
the second equality of \eqref{limitlambda}.
This is obtained as a corollary of the following proposition.
%
%%%%%%%%%%%
%%% t2k %%%
%%%%%%%%%%%
%
\begin{prop} \label{t2k}
For each $\nu\in P^+$, there exists $\ell \ge 0$ such that
for all $(s_{i,j})\in \BT(\lambda_2)^\nu_{\lambda_1}$,
$$s_{1,1}\ge s_{2,2}+\ell, \qquad s_{1,3}\ge a-\ell.$$
\end{prop}

\begin{proof}
Suppose that there exists $(t_{i,j})\in \BT(\lambda_2)^\nu_{\lambda_1}$
such that either $t_{1,3}=\lambda_1(h_2)$ or $t_{1,3}+t_{2,3}=\lambda_1(h_2)+t_{1,2}$.
Then, $\ell=0$ satisfies the condition of the proposition.
Indeed, let $(s_{i,j}) \in \BT(\lambda_2)^\nu_{\lambda_1}$.
Then, $s_{1,1} \geq s_{2,2} + 0$ is true
by the first inequality of \eqref{sst2}. Also we see by
the third equality of \eqref{sameweight} that
$$
 s_{1,3}+s_{2,3}=t_{1,3}+t_{2,3}\ge \lambda_1(h_2).
$$
Since $s_{2,3}\le \lambda_2(h_2)$ by \eqref{sst2}, and
$\lambda_1(h_2)-\lambda_2(h_2) + \lambda_1(h_1) - \lambda_2(h_1) \ge 0$ by
the fact that $\bos_{1}(\lambda_{1}-\lambda_{2})(h_{2}) \ge 0$
(recall that $\bos_{1}(\lambda_{1}-\lambda_{2})(h_{2}) \in P^{+}$), we get
$$
 s_{1,3}\ge \lambda_1(h_2)-s_{2,3} \ge
 \lambda_1(h_2)-\lambda_2(h_2) \ge
 \lambda_2(h_1)-\lambda_1(h_1)\ge a = a + 0.
$$

Consider now the case when for all $(t_{i,j}) \in \BT(\lambda_2)^\nu_{\lambda_1}$,
both of $t_{1,3} < \lambda_1(h_2)$ and $t_{1,3}+t_{2,3} < \lambda_1(h_2)+t_{1,2}$ hold.
Since $t_{1,1}\ge t_{2,2}$ by \eqref{sst1},
we can choose $\ell \ge 0$ minimal with the property
that $t_{1,1}\ge t_{2,2}+\ell$ for all $(t_{i,j}) \in \BT(\lambda_2)^\nu_{\lambda_1}$.
If $\ell \ge a$, then the statement of the proposition is trivially true.
Assume now that $\ell < a$, and
suppose that there exists $(t_{i,j})$ with $t_{1,3} < a-\ell$.
Fix $(s_{i,j}) \in \BT(\lambda_2)^{\nu}_{\lambda_1}$
such that $s_{1,1}=s_{2,2}+\ell$.
Since both of $(s_{i,j})$ and $(t_{i,j})$ are
elements of $\BT(\lambda_2)^{\nu}$, we have
by \eqref{sst1} and \eqref{sameweight}
$$
 t_{1,2} + (a - \ell) > t_{1,2} + t_{1,3} =
 \lambda_{2}(h_{1})+\lambda_{2}(h_{2})-t_{11}=
 \lambda_{2}(h_{1})+\lambda_{2}(h_{2})-s_{11}=
 s_{1,2} + s_{1,3}.
$$
Hence we get
\begin{align*}
\lambda_1(h_1)
 & \geq t_{1,2} \quad
   \text{by the first inequality of \eqref{LR1}} \\
 & > s_{1,2} +s_{1,3} - (a - \ell) \\
 & = (\lambda_2(h_1) + \lambda_2(h_2) - s_{1,1}) + \ell - a \quad
   \text{by the first equality of \eqref{sst1}} \\
 & = \lambda_2(h_1) + \lambda_2(h_2) - s_{2,2}  - a \\
 & \geq \lambda_2(h_1) + \lambda_2(h_2) - \lambda_2(h_2)  - a \quad
   \text{by the second equality of \eqref{sst1}} \\
 & = \lambda_2(h_1) - a.
\end{align*}
So, $\lambda_1(h_1) > \lambda_2(h_1) - a$, which gives $a > \lambda_2(h_1) - \lambda_1(h_1)$, which is a contradiction.
Hence $t_{1,3} \geq a - \ell$ for all $(t_{i,j}) \in \BT(\lambda_2)^\nu_{\lambda_1}$
and the proof is complete.
\end{proof}

The following corollary is now trivially checked using Subsection~\ref{nak}.
Thus we have proved that Theorem~\ref{mainthm}(iii) is true
if $\blambda$ and $\bmu$ satisfy the conditions \eqref{limitlambda} and
\eqref{limitbmu} with $\bow =\bos_{1}$ (see \eqref{eq:t2k}).
%
%%%%%%%%%%%%%%%
%%% cor-t2k %%%
%%%%%%%%%%%%%%%
%
\begin{cor} \label{cor-t2k}
Keep the notation and setting in the proposition above.
Let $\nu \in P^{+}$, and let $\ell$ be as in the proposition above. Then, the assignment
$(s_{i,j}) \mapsto (s'_{i,j})$,
$$s'_{1,1} = s_{1,1}, \qquad s'_{1,2} = s_{1,2} + (a - \ell), \qquad s'_{1,3} = s_{1,3} - (a - \ell)$$
$$s'_{2,2} = s_{2,2} + \ell, \qquad s'_{2,3} = s_{2,3} + (a - \ell),$$
defines an injective map $\BT(\lambda_2)^\nu_{\lambda_1} \hookrightarrow
\BT(\lambda_2 + a(\omega_2-\omega_1))^\nu_{\lambda_1 - a(\omega_2-\omega_1)}$.
\end{cor}

\bibliographystyle{alpha}
\bibliography{cfs-schur-biblist}

\end{document}